\documentclass[a4paper,10pt]{amsart}
\usepackage[utf8]{inputenc}
\usepackage{amsthm}
\usepackage{amsmath}
\usepackage{bm}
\usepackage{amsfonts}
\usepackage{amssymb}
\usepackage{mathtools}
\usepackage{mathrsfs}
\usepackage{setspace}
\usepackage{xcolor}
\usepackage{textgreek}
\usepackage{todonotes}
\usepackage{thmtools} 

\usepackage[pdfdisplaydoctitle,colorlinks,breaklinks,urlcolor=blue,linkcolor=blue,citecolor=blue]{hyperref} 

\newtheorem{thm}{Theorem}[section]

\newtheorem{definition}[thm]{Definition}

\newtheorem{lem}[thm]{Lemma}

\newtheorem{prop}[thm]{Proposition}

\newcommand{\T}{{\mathbb{T}^2}}
\newcommand{\N}{\mathbb{N}}
\newcommand{\R}{\mathbb{R}}
\newcommand{\Z}{\mathbb{Z}}
\newcommand{\PP}{\mathbb{P}}
\newcommand{\E}[1]{\mathbb{E}\left[#1\right]}
\newcommand{\tE}[1]{\tilde{\mathbb{E}}\left[#1\right]}
\newcommand{\hE}[1]{\hat{\mathbb{E}}\left[#1\right]}

\newcommand{\dvg}{\mbox{div}\,}

\theoremstyle{remark}
\newtheorem{rmk}[thm]{Remark}

\numberwithin{equation}{section}

\title[From additive to transport noise]{From additive to transport noise \\ in 2D fluid dynamics}
\author[F. Flandoli]{Franco Flandoli}
  \address{Scuola Normale Superiore, Piazza dei Cavalieri, 7, 56126 Pisa, Italia}
 \email{\href{mailto:franco.flandoli@sns.it}{franco.flandoli@sns.it}}
\author[U. Pappalettera]{Umberto Pappalettera}
  \address{Scuola Normale Superiore, Piazza dei Cavalieri, 7, 56126 Pisa, Italia}
 \email{\href{mailto:umberto.pappalettera@sns.it}{umberto.pappalettera@sns.it}}
\date\today

\begin{document}

\begin{abstract}
Additive noise in Partial Differential equations, in particular those of fluid mechanics, has relatively natural motivations.
The aim of this work is showing that suitable multiscale arguments lead rigorously, from a model of fluid with additive noise, to transport type noise. 
The arguments apply both to small-scale random perturbations of the fluid acting on a large-scale passive scalar and to the action of the former on the large scales of the fluid itself.
Our approach consists in studying the (stochastic) characteristics associated to small-scale random perturbations of the fluid, here modelled by stochastic 2D Euler equations with additive noise, and their convergence in the infinite scale separation limit. 
\end{abstract}

\maketitle


\section{Introduction} \label{sec:intro}
Let $T > 0$ be fixed. In this work we are concerned with convergence of characteristics associated with stochastic Euler equations in vorticity form on the two-dimensional torus $\T \coloneqq \R^2/(2\pi \Z^2)$:
\begin{align} \label{eq:euler_intro}
d \xi^\epsilon_t + (v^\epsilon_t+u^\epsilon_t) \cdot \nabla \xi^\epsilon_t dt
=
- \epsilon^{-1} \xi^\epsilon_t dt + \epsilon^{-1} \sum_{k \in \N} \varsigma_k dW^k_t,
\quad t \in [0,T],
\end{align}
where $\xi^\epsilon$ is the zero-mean unknown vorticity field, $u^\epsilon$ is the velocity field reconstructed from $\xi^\epsilon$ via the Biot-Savart kernel: $u^\epsilon_t = -\nabla^\perp (-\Delta)^{-1} \xi^\epsilon_t$, $v^\epsilon$ is a divergence-free external field with suitable regularity, $\varsigma_k : \T \to \R$ with zero average for every $k \in \N$, $(W^k)_{k \in \N}$ is a family of i.i.d. Wiener processes defined on a filtered probability space $(\Omega,\mathcal{F}_t,\PP)$, and $\epsilon \ll 1$ is a scaling parameter.  

Equations \eqref{eq:euler_intro} above aim to represent the small-scale component of a two-dimensional incompressible fluid \cite{BoEc12}, with the additive noise and damping on the right-hand-side modelling the influence on the fluid of a possibly irregular boundary or topography.
The choice of the parameter $\epsilon^{-1}$ in front of both noise and damping is appropriate when looking at the system with respect to the point of view of a large-scale observer, see \cite{FlPa21} and \autoref{ssec:motivations} for details. 
In view of this, it makes sense to couple \eqref{eq:euler_intro} with a large-scale scalar dynamics:
\begin{align} \label{eq:large_eps_intro}
d \Xi^\epsilon_t + (v^\epsilon_t + u^\epsilon_t) \cdot \nabla \Xi^\epsilon_t dt
=
\nu \Delta \Xi^\epsilon_t dt
+ q^\epsilon_t dt,
\quad t \in [0,T],
\end{align}
either passive (in which case the external field $v^\epsilon$ should be interpreted as given a priori) or active (in which case the external field $v^\epsilon$ could depend on the large-scale dynamics itself, as for instance in the vorticity formulation of 2D Navier-Stokes equations, where $v^\epsilon_t = -\nabla^\perp (-\Delta)^{-1} \Xi^\epsilon_t$). 

In \eqref{eq:large_eps_intro} above, $\nu\geq 0$ is a fixed parameter that represents molecular diffusivity (passive dynamics) or viscosity (active dynamics), and $q^\epsilon$ is a given source term with suitable integrability. 

Let $(\tilde{\Omega},\tilde{\mathcal{F}}_t,\tilde{\PP})$ be an auxiliary probability space and let $w$ be a standard $\R^2$-valued Wiener process defined on $(\tilde{\Omega},\tilde{\mathcal{F}}_t,\tilde{\PP})$.  
The (stochastic) characteristics $\phi^\epsilon$ associated with problem \eqref{eq:euler_intro} are given by the family of maps $\phi^\epsilon_t:\T \to \T$ satisfying
\begin{align} \label{eq:char_eps_intro}
\phi^\epsilon_t(x)
&=
x
+
\int_0^t
v^\epsilon_s(\phi^\epsilon_s(x)) ds 
+
\int_0^t
u^\epsilon_s(\phi^\epsilon_s(x)) ds +
\sqrt{2\nu} w_t,
\end{align}
where $t \in [0,T]$, $x \in \T$.
Since $v^\epsilon$ and $u^\epsilon$ are divergence-free and have sufficient regularity, the characteristics $\phi^\epsilon$ defined above constitute a \emph{stochastic flow of measure-preserving homeomorphisms}, in the sense of \autoref{def:stoch_flow} below.
The interest in studying the solution of \eqref{eq:char_eps_intro} is motivated by the following representation formula for the solution of \eqref{eq:large_eps_intro}:
\begin{align} \label{eq:repr_large_eps}
\Xi^\epsilon_t
&=
\tE{ \Xi_0 \circ (\phi^\epsilon_t)^{-1} + \int_0^t q^\epsilon_s \circ \phi^\epsilon_s \circ(\phi^\epsilon_t)^{-1}ds},
\end{align} 
where $\tilde{\mathbb{E}}$ is the expectation on $\tilde{\Omega}$ with respect to $\tilde{\PP}$ and we have tacitly assumed that the initial condition $\Xi^\epsilon|_{t=0} = \Xi_0$ is independent of $\epsilon$. See \autoref{def:sol} for more details on the notion of solution adopted in the present paper.

The main purpose of this work - cfr. \autoref{thm:char} - is to investigate conditions allowing to prove convergence in a suitable sense, as $\epsilon \to 0$, of $\phi^\epsilon$ towards the solution of: 
\begin{align} \label{eq:char_intro}
\phi_t(x)
&=
x
+
\int_0^t
v_s(\phi_s(x)) ds 
+
\sum_{k \in \N} \int_0^t \sigma_k(\phi_s(x)) \circ dW^k_s +
\sqrt{2\nu} w_t,
\end{align}
where $\sigma_k = -\nabla^\perp (-\Delta)^{-1} \varsigma_k$ and $v^\epsilon \to v$ in a certain sense.

The notion of convergence $\phi^\epsilon \to \phi$ contained in \autoref{thm:char} permits to prove a notion of weak convergence of the large-scale osservable $\Xi^\epsilon$ given by \eqref{eq:repr_large_eps} towards 
\begin{align} \label{eq:repr_large}
\Xi_t
&=
\tE{ \Xi_0 \circ (\phi_t)^{-1} + \int_0^t q_s \circ \phi_s \circ(\phi_t)^{-1}ds},
\end{align}
that solves the large-scale dynamics with \emph{transport noise}:
\begin{align} \label{eq:large_intro}
d \Xi_t + v_t \cdot \nabla \Xi_t dt
+ \sum_{k \in \N} \sigma_k \cdot \nabla \Xi_t \circ dW^k_t
&=
\nu \Delta \Xi_t dt + q_t dt,
\end{align}
where $v$ is independent of $\Xi$ for passive dynamics, while it could depend of $\Xi$ itself for active dynamics, and $q^\epsilon \to q$ in a sense to be specified later. The precise meaning of weak convergence is made rigorous in \autoref{thm:conv_large} below.

%
%

We think that these results are fundamental for a proper interpretation of transport noise in SPDEs, at least for the two classes considered here. 
Several papers considered transport noise so far, either in passive scalars
(\cite{Dolgo}, \cite{FlaGaleLuoPTRSA}, 
\cite{Galeati}, \cite{Gess}, \cite{LeJan Ray}, \cite{MajdaK},
\cite{Sreen}), passive vector fields (\cite{FlaMaurNek}, \cite{FlaOliv
advection}, \cite{Krause Radler}, \cite{Zeldovich}) and fluid mechanics
equations themselves (\cite{BCF 91 mult noise}, \cite{BCF 92 mult noise},
\cite{BrFlMa16}, \cite{BrSl20+}, \cite{CrFlHo19},
\cite{Cruzeiro Torr}, \cite{DrivasHolm}, \cite{DrivasHolmLehaly},
\cite{FlaGaleLuoJEE}, \cite{FlaGalLuorate}, \cite{FlaPappWater},
\cite{Funaki}, \cite{Holm}, \cite{MiRo04}, \cite{MiRo05}, \cite{Yokohama}).
In terms of consequences of transport noise, among the aforementioned works are proved several results concerning well-posedness, enhanced dissipation and mixing properties of fluid dynamics equations perturbed by transport noise, thus being a good starting point towards a rigorous understanding of turbulence in fluids.
However, unlike the case of additive noise, that is widely accepted as a source of randomness, transport noise needs a more careful justification.
One possible attempt is given by Wong-Zakai approximation results, largely investigated both in and outside the realm of fluid dynamics (\cite{BrCaFl90}, \cite{Gy88}, \cite{Gy89},
\cite{HoLeNi19}, \cite{HoLeNi19+}, \cite{TeZa06}, \cite{Tw93}).

Let us explain what is added to these works by the present paper. Concerning the passive dynamics, several Wong-Zakai type results of convergence to the white noise transport in Stratonovich form have been proved before (see also \cite{Pa21+} for a dissipation enhancement result due to the presence of a Stratonovich-to-It\=o  corrector), but this seems to be the first work where the velocity field approximating the white noise one is the solution of a nonlinear fluid mechanics equation. 
Concerning the active dynamics, the results contained in this paper extend and make more precise our previous work \cite{FlPa21}: i) some details in the proof of \cite[Proposition 4.1]{FlPa21}, which after publication appeared not precise, are fixed here in \autoref{thm:char}; ii) more importantly, the term $u_{t}^{\epsilon}\cdot\nabla\xi_{t}^{\epsilon}$ was
absent in \cite{FlPa21}, which therefore should be interpreted more along the research lines of model reduction, inspired by \cite{MaTiVE01}, instead of multiscale analysis of the full problem.

\subsection{Examples} \label{ssec:ex}
Throughout the paper we keep ourselves in a setting as general as possible, in order to comprehend, in our abstract results, the greatest number of particular cases.
However, our work has been motivated by two main examples:
\begin{itemize}
\item \emph{Advection-diffusion equation}. 
Consider the following system, describing the evolution of the concentration $\rho^\epsilon$ of a passive scalar advected by the Euler flow and subject to the influence of an external source $q^\epsilon$:
\begin{align*}
\begin{cases}
d\rho^\epsilon_t + (v_t + u^\epsilon_t) \cdot \nabla \rho^\epsilon_t dt
=
\nu \Delta \rho^\epsilon_t dt + q^\epsilon_t dt,
\\
d \xi^\epsilon_t + (v_t+u^\epsilon_t) \cdot \nabla \xi^\epsilon_t dt
=
- \epsilon^{-1} \xi^\epsilon_t dt + \epsilon^{-1} \sum_{k \in \N} \varsigma_k dW^k_t,
\\
u^\epsilon_t = -\nabla^\perp(-\Delta)^{-1}\xi^\epsilon_t.
\end{cases}
\end{align*}
We have taken $\nu\geq 0$ and $v^\epsilon=v$, independent of $\epsilon$, since the passive scalar does not affect the external field. 
In this setting, $\rho^\epsilon$ converges towards the solution of the limiting advection-diffusion equation with transport noise:
\begin{align*}
d\rho_t + v_t \cdot \nabla \rho_t dt
+ \sum_{k \in \N} \sigma_k \cdot \nabla \rho_t \circ dW^k_t
&=
\nu \Delta \rho_t dt + q_t dt.
\end{align*}

\item \emph{Navier-Stokes and Euler equations}.
Consider the following system, describing the coupling between large-scale Navier-Stokes ($\nu > 0$) or Euler ($\nu=0$) equations and small-scale stochastic Euler equations:
\begin{align*}
\begin{cases}
d\Xi^\epsilon_t + (v^\epsilon_t + u^\epsilon_t) \cdot \nabla \Xi^\epsilon_t dt
=
\nu \Delta \Xi^\epsilon_t dt + q^\epsilon_t dt,
\\
d \xi^\epsilon_t + (v^\epsilon_t+u^\epsilon_t) \cdot \nabla \xi^\epsilon_t dt
=
- \epsilon^{-1} \xi^\epsilon_t dt + \epsilon^{-1} \sum_{k \in \N} \varsigma_k dW^k_t,
\\
v^\epsilon_t = -\nabla^\perp(-\Delta)^{-1}\Xi^\epsilon_t,
\\
u^\epsilon_t = -\nabla^\perp(-\Delta)^{-1}\xi^\epsilon_t.
\end{cases}
\end{align*}
We take $q^\epsilon$ and $\Xi_0$ with zero spatial average, so that $\Xi^\epsilon$ is zero mean, too.
Notice that in this case the field $v^\epsilon$ is generated by $\Xi^\epsilon$ itself through the Biot-Savart law $v^\epsilon_t = -\nabla^\perp(-\Delta)^{-1}\Xi^\epsilon_t$, in particular $v^\epsilon$ is random. On the other hand, the external source $q^\epsilon$ can be thought as given a priori and deterministic.
In this setting, $\Xi^\epsilon$ converges towards the solution of the limiting Navier-Stokes or Euler equations with transport noise:
\begin{align*}
\begin{cases}
d \Xi_t + v_t \cdot \nabla \Xi_t dt
+ \sum_{k \in \N} \sigma_k \cdot \nabla \Xi_t \circ dW^k_t
=
\nu \Delta \Xi_t dt + q_t dt,
\\
v_t=-\nabla^\perp(-\Delta)^{-1}\Xi_t .
\end{cases}
\end{align*}
It is worth of mention that, also in the limit, the velocity field $v$ is still generated by $\Xi$ through the Biot-Savart law $v_t = -\nabla^\perp(-\Delta)^{-1}\Xi_t$.
\end{itemize} 

\subsection{Motivations} \label{ssec:motivations}
As already mentioned in the Introduction, \eqref{eq:euler_intro} aims to represent the small-scale component of a two-dimensional incompressible fluid, looked at by a large-scale observer.
At large scales the fluid shows a turbulent behaviour, and its statistical properties are well-described by solutions of stochastic equations, although the underlying continuum mechanics equations that govern the evolution of the fluid are deterministic. 

We refer to \cite[Chapter 2]{FlPa21} and reference therein for a complete discussion about the equations under investigation in this paper and the interest for their asymptotical behaviour as $\epsilon \to 0$.

\subsubsection{On the additive noise and damping}
Additive noise in SPDEs is so common that apparently we do not need a justification for introducing it, as we have done in equation \eqref{eq:euler_intro} above. 
However, a short discussion may help to convince ourselves that it is very natural, and moreover to understand that also the
damping term is needed. 

Our opinion, described more extensively in the proposal of \cite[Chapter 1]{Fla libro}, is that an additive noise is a good compromise to keep into account the vortices produced by obstacles and irregularities at the boundary or internal obstacles, which are not explicitly described in the
mathematical formulation, often based on the torus geometry or a domain with smooth boundary. 
Such obstacles introduce vortices, eddies, that could be idealized and described as a jump Markov process $W_{N}(t)$ in the Hilbert space $H$ of $L^{2}(\T)$ vorticity fields on the torus; the fluid equation perturbed by the creation of these new vortices takes a priori the form
\[
\partial_{t}\xi_{t}+(v_{t}+u_{t})\cdot\nabla\xi_{t}=\partial_{t}W_{N}(t)
\]
where $\partial_{t}W_{N}(t)$ is a sum of delta Dirac in time, with the effect that $\xi_{t}$ jumps at those times, namely (if $t_{i}$ denotes one of such times) $\xi_{t_{i}^{+}}$ is equal to $\xi_{t_{i}^{-}}$ plus the created vortex. 
We have indexed $W_{N}(t)$ by $N$ to anticipate that we consider a regime with frequent creation of vortices of
small amplitude. 

Scaling the parameters of $W_{N}(t)$ in the right way, see \cite{Fla libro}, under suitable assumptions of zero average of $W_{N}(t)$ and integrability, $W_{N}(t)$ converges in law to a Brownian motion $W(t)$ in $H$ with a suitable covariance. 
This is our motivation for the equation with
additive noise
\[
d\xi_{t}+(v_{t}+u_{t})\cdot\nabla\xi_{t}dt=dW(t)  .
\]
However, as it is easily seen by It\^{o} formula \cite[Chapter 2]{Fla libro}, such additive noise introduces systematically energy, fact that is not acceptable from the physical viewpoint: the vortices created by obstacles do not increase the energy (at most, some energy is lost in thermal
dissipation at the boundary). 
Therefore some sort of compensation is needed. 
The simplest is to think that the forces which are responsible for the creation of vortices by the obstacles are somewhat similar to a friction. Thus we introduce a
friction term to maintain equilibrium:
\[
d\xi_{t}+(v_{t}+u_{t})\cdot\nabla\xi_{t}dt=-\lambda\xi_{t}dt+dW(t).
\]
This is the origin of the fluid model. 
The particular scaling attributed above to the terms $-\lambda\xi_{t}dt$ and $dW(t)$ is related to
a different argument, which is explained in the next paragraph.

\subsubsection{On the parameter $\epsilon^{-1}$}
An important feature of \eqref{eq:euler_intro} is the presence of the scaling parameter $\epsilon^{-1}$ in front of \emph{both} noise and damping, in contrast to the widely-studied diffusive scaling given by coefficients $\epsilon^{-1}$ in front of the damping and $\epsilon^{-1/2}$ in front of the noise.
Let us recall briefly where this scaling comes from, referring to \cite{FlPa21} for further details.
 
We suppose to have a small time scale $\mathcal{T}_{\text{S}}$, at which we observe the vorticity field $\xi$. At this scale, the small scales evolve according to deterministic equations, and the typical intensity and turnover time of $\xi$ are of order one. 

Let us now take an intermediate point of view on the system, say human-scale, $\mathcal{T}_{\text{M}} \coloneqq \epsilon^{-1} \mathcal{T}_{\text{S}}$. At this scale, fluctuations of $\xi = \xi^\epsilon$ look random and could be well modeled by stochastic equations \eqref{eq:euler_intro}, with the crucial difference of a coefficient $\epsilon^{-1/2}$ in front of the noise rather than $\epsilon^{-1}$.

Only when we look at the system with respect to a large time scale $\mathcal{T}_{\text{L}} \coloneqq \epsilon^{-1} \mathcal{T}_{\text{M}}$ the scaling of \eqref{eq:euler_intro} appears. 
As a result of the theory here developed, under this point of view the small scale fluctuations behave as a white noise of multiplicative type.

We remark that, in our arguments, spatial scaling is less important then temporal scaling. 
As it emerges from computations performed in \cite[Subsection 2.3]{FlPa21}, the spatial scaling only affects the spatial covariance of the noise in \eqref{eq:euler_intro}.
For the sake of concreteness, suppose that $W_{\text{M}}(\tilde{t},\tilde{x})$ is the noise perturbing $\xi^\epsilon$ at intermediate scales and $W_{\text{L}}(t,x)$ is the noise perturbing $\xi^\epsilon$ at large scales, with mesoscopic variables $\tilde{t},\tilde{x}$ related to macroscopic variables $t,x$ by the formulas $\tilde{t} = \epsilon^{-1} t$, $\tilde{x} = \epsilon_X^{-1} x$. Then it holds the equality in law
\begin{align*}
W_{\text{M}}(\tilde{t},\tilde{x})
=
\epsilon^{1/2} W_{\text{L}}\left(t,\epsilon^{-1}_X x\right).
\end{align*}
Moreover, assuming that the elements producing the noise (topography, boundaries \emph{et cetera}) are actually large-scale, we can suppose that the covariance of $W_{\text{M}}$ is slowly-varying with respects to $\tilde{x}$, or equivalently
\begin{align*}
W_{\text{L}}\left(t,\epsilon^{-1}_X x\right)
=
\sum_{k \in \N}
\varsigma_k(x) W^k_t,
\end{align*}
with $\varsigma_k$ and $W^k$ as in \eqref{eq:euler_intro}.

\subsection{Structure of the paper}
In \autoref{sec:notations} we introduce some notation and recall classical results that will be frequently used in the remainder of the paper. This section contains, among others: main properties of the Biot-Savart kernel on the torus $-\nabla^\perp(-\Delta)^{-1}$; a useful Gronwall-type lemma for ODEs with log-Lipschitz drift; notions of solution and well-posedness results for stochastic Euler equations \eqref{eq:euler_intro}, equations of characteristics \eqref{eq:char_eps_intro} and \eqref{eq:char_intro}, and large-scale dynamics \eqref{eq:large_eps_intro} and \eqref{eq:large_intro}.
Also, here we introduce our main working assumptions (A1)-(A7), and in the last part of this section we state our two main results, concerning convergence of characteristics (\autoref{thm:char}) and subsequent convergence of large-scale dynamics (\autoref{thm:conv_large}).

In the first part of \autoref{sec:tech}, we define a linearized version of \eqref{eq:euler_intro}, where we neglect the nonlinear term. This approach is similar to that of \cite{FlPa21}, and the key idea is that, although the solution $\theta^\epsilon$ of linearized equation is not close ot the actual solution $\xi^\epsilon$ of \eqref{eq:euler_intro}, the characteristics generated by $\theta^\epsilon$ are close to the characteristics generated by $\xi^\epsilon$, in particular they have the same limit as $\epsilon \to 0$.

In the same section we present two main technical results, needed in the proof of \autoref{thm:char}. 
The first of those results is \autoref{prop:old}, which ensures that the linear part $\theta^\epsilon$ of the small-scale dynamics behaves as a Stratonovich white-in-time noise as $\epsilon \to 0$, at least in a distributional sense.
The second result \autoref{prop:sup_z}, instead, aims to rigorously prove the closeness of the characteristics generated by $\theta^\epsilon$ and $\xi^\epsilon$, and it is one of the main novelties of this paper with respect to \cite{FlPa21}.

The proof of \autoref{thm:char} is contained in \autoref{sec:conv_char}, and it is based on a Gronwall-type lemma and It\=o Formula applied to a smooth approximation $g_\delta(x)$ of the absolute value $|x|$, $x \in \R^2$.
The proof of \autoref{thm:conv_large} can be found in \autoref{sec:conv_large}, and it relies on representation formulas \eqref{eq:repr_large_eps} and \eqref{eq:repr_large} and a measure-theoretic argument.

Finally, in \autoref{sec:ex} we discuss how our main motivational examples - cfr. \autoref{ssec:ex} - fit our abstract setting.
In particular, the non-trivial one is the coupled system given by deterministic Navier-Stokes equations at large scales plus stochastic Euler equations at small scales; we identify an additional but very natural condition (A8) on the limit external source $q$ that allows to verify assumptions (A1)-(A7) for the system under consideration.

\section{Notations, preliminaries and main results} \label{sec:notations}

In this section we collect definitions, notations and classical results needed in the paper. Also, we introduce our main working assumptions (A1)-(A7), and state our main results.  

\subsection{Properties of the Biot-Savart kernel}

Here we briefly recall some useful properties of the Biot-Savart kernel $K$. We refer to \cite{MaPu94,BrFlMa16} for details and proofs.

First of all, the Biot-Savart kernel $K$ is defined as $K=-\nabla^\perp G = (\partial_2 G, - \partial_1 G)$, where $G$ is the Green function of the Laplace operator on the torus $\T$ with zero mean.

For $p \in (1,\infty)$ and $\xi \in L^p(\mathbb{T}^2)$ with zero-mean, the convolution with $K$ represents the Biot-Savart operator:
\begin{align*}
K \ast \xi = -\nabla^\perp (-\Delta)^{-1} \xi,
\end{align*}
that to every zero-mean $\xi \in L^p(\mathbb{T}^2)$ associates the unique zero-mean, divergence-free
velocity vector field $u \in W^{1,p}(\mathbb{T}^2,\mathbb{R}^2)$ such that $\mbox{curl}\,u = \xi$.
Moreover, for every $p \in (1,\infty)$ there exist constants $c$, $C$ such that for every zero-mean $\xi \in L^p(\mathbb{T}^2)$
\begin{align*}
c\| \xi \|_{L^p(\mathbb{T}^2)} 
\leq 
\| K \ast \xi \|_{W^{1,p}(\mathbb{T}^2,\mathbb{R}^2)}
\leq
C\| \xi \|_{L^p(\mathbb{T}^2)}.
\end{align*}

Also, recall that since $K \in L^1(\T,\R^2)$ the convolution $K \ast \xi$ is well-defined for every $\xi \in L^p(\mathbb{T}^2)$, $p \in [1,\infty]$ and the following estimate holds: 
\begin{align} \label{eq:K}
\| K \ast \xi \|_{L^p(\mathbb{T}^2,\mathbb{R}^2)}
\leq 
\| K \|_{L^1(\mathbb{T}^2,\mathbb{R}^2)}
\| \xi \|_{L^p(\mathbb{T}^2)}.
\end{align}

Let $r \geq 0$. Denote $\gamma:[0,\infty) \to \R$ the concave function:
\begin{equation*}
\gamma(r) =
r(1-\log r) \mathbf{1}_{\{0<r<1/e\}} + 
(r+1/e)		\mathbf{1}_{\{r\geq 1/e\}}.
\end{equation*}
The following two lemmas are proved in \cite{MaPu94} and \cite{BrFlMa16}.

\begin{lem} \label{lem:log_lip}
There exists a constant $C$ such that:
\begin{equation*}
\int_{\mathbb{T}^2}
\left| K(x-y) - K(x'-y)\right| dy
\leq
C \gamma(|x-x'|)
\end{equation*}
for every $x,x' \in \mathbb{T}^2$.
\end{lem}

\begin{lem} \label{lem:comp}
Let $T>0$, $\lambda>0$, $a_0 \in [0,\exp(1-2e^{\lambda T})]$ be constants.
Let $a:[0,T] \to \R$ be such that for every $t \in [0,T]$:
\begin{align*}
a_t \leq a_0 + \lambda  \int_0^t \gamma(a_s) ds.
\end{align*}
Then for every $t \in [0,T]$ the following estimate holds:
\begin{align*}
a_t \leq e a_0^{\exp(-\lambda t)}.
\end{align*}
\end{lem}

\subsection{Stochastic flows of measure-preserving homeomorphisms}

As a convention, in the following we say that $\mathcal{N} \subset \Omega$ (respectively $\tilde{\mathcal{N}} \subset \tilde{\Omega}$) is \emph{negligible} if it is measurable and $\PP(\mathcal{N}) = 0$ (respectively $\tilde{\PP}(\tilde{\mathcal{N}}) = 0$), without explicit mention of the reference probability measure.
Unless otherwise specified, we will always denote with $\mathcal{N}$ negligible sets in $\Omega$, and with $\tilde{\mathcal{N}}$ negligible sets in $\tilde{\Omega}$.

Let us begin this paragraph with the following fundamental definition.
\begin{definition} \label{def:stoch_flow}
A measurable map $\phi: \Omega \times \tilde{\Omega} \times [0,T] \times \T \to \T$ is a \emph{stochastic flow of measure-preserving homeomorphisms} provided there exist negligible sets $\mathcal{N} \subset \Omega$ and $\tilde{\mathcal{N}} \subset \tilde{\Omega}$ such that:
\begin{itemize}
\item
for every $\omega \in \mathcal{N}^c$, $\tilde{\omega} \in \tilde{\mathcal{N}}^c$ and $t \in [0,T]$, the map $\phi(\omega,\tilde{\omega},t,\cdot):\T \to \T$ is a homeomorphism of the torus and
\begin{align*}
\int_\T f(x) dx = \int_\T f(\phi(\omega,\tilde{\omega},t,y)) dy
\end{align*}
for every $f \in L^1(\T)$;
\item
for every $\tilde{\omega} \in \tilde{\mathcal{N}}^c$ and $x \in \T$, the stochastic process $\phi(\cdot,\tilde{\omega},\cdot,x):\Omega \times [0,T] \to \T$ is progressively measurable with respect to the filtration $(\mathcal{F}_t)_{t \in [0,T]}$.
\end{itemize}
\end{definition}
In some circumstances it can be useful to have the following: 
\begin{definition}
A stochastic flow of measure-preserving homeomorphisms $\phi$ is called \emph{inviscid} if there exist negligible sets $\mathcal{N} \subset \Omega$ and $\tilde{\mathcal{N}} \subset \tilde{\Omega}$, and a measurable map $\psi:\Omega \times [0,T] \times \T \to \T$ such that for every $\omega \in \mathcal{N}^c$, $\tilde{\omega} \in \tilde{\mathcal{N}}^c$, $t \in [0,T]$ and $x \in \T$
\begin{align*}
\phi(\omega,\tilde{\omega},t,x) = \psi(\omega,t,x).
\end{align*}
\end{definition}

With a little abuse of notation, hereafter we identify an inviscid stochastic flow of measure-preserving homeomorphisms $\phi$ with its $\tilde{\omega}$-independent representative $\psi$. 

Let us now clarify the meaning of \eqref{eq:char_eps_intro}, \eqref{eq:char_intro}.

A measurable map  $\phi^\epsilon:\Omega \times \tilde{\Omega} \times [0,T] \times \T \to \T$ is a solution of \eqref{eq:char_eps_intro} if there exist negligible sets $\mathcal{N} \subset \Omega$ and $\tilde{\mathcal{N}} \subset \tilde{\Omega}$ such that for every $\omega \in \mathcal{N}^c$, $\tilde{\omega} \in \tilde{\mathcal{N}}^c$, $t \in [0,T]$ and $x \in \T$:
\begin{align*}
\phi^\epsilon(\omega,\tilde{\omega},t,x)
&=
x
+
\int_0^t
v^\epsilon(\omega,s,\phi^\epsilon(\omega,\tilde{\omega},s,x)) ds \\
&\quad+
\int_0^t
u^\epsilon(\omega,s,\phi^\epsilon(\omega,\tilde{\omega},s,x)) ds +
\sqrt{2\nu} w(\tilde{\omega},t),
\end{align*}
where the previous identity can be interpreted as an equation on $\T$ since one can check $\phi^\epsilon(\omega,\tilde{\omega},t,x+2\pi\mathbf{e}) = \phi^\epsilon(\omega,\tilde{\omega},t,x)+2\pi\mathbf{e}$ for $\mathbf{e}=(1,0)$ and $\mathbf{e}=(0,1)$.

Similarly, a measurable map  $\phi:\Omega \times \tilde{\Omega} \times [0,T] \times \T \to \T$ is a solution of \eqref{eq:char_intro} if there exist negligible sets $\mathcal{N} \subset \Omega$ and $\tilde{\mathcal{N}} \subset \tilde{\Omega}$ such that for every $\tilde{\omega} \in \tilde{\mathcal{N}}^c$ and $x \in \T$, the stochastic process $\phi(\cdot,\tilde{\omega},\cdot,x):\Omega \times [0,T] \to \T$ is progressively measurable with respect to the filtration $(\mathcal{F}_t)_{t \in [0,T]}$, and for every $\omega \in \mathcal{N}^c$, $\tilde{\omega} \in \tilde{\mathcal{N}}^c$, $t \in [0,T]$ and $x \in \T$:
\begin{align*}
\phi(\omega,\tilde{\omega},t,x)
&=
x
+
\int_0^t
v(\omega,s,\phi(\omega,\tilde{\omega},s,x)) ds \\
&\quad+
\sum_{k \in \N} \left(\int_0^t \sigma_k(\phi(\cdot,\tilde{\omega},s,x))) \circ dW^k_s \right)(\omega) +
\sqrt{2\nu} w(\tilde{\omega},t).
\end{align*}
Notice that progressive measurability of the process $\phi(\cdot,\tilde{\omega},\cdot,x):\Omega \times [0,T] \to \T$ is necessary to make sense of the Stratonovich stochastic integral apparing in the equation above. 

\subsection{Notions of solution and some well-posedness results}

\subsubsection{Well-posedness of small-scale dynamics and characteristics}

First we make the following assumptions on the external fields:  
\begin{itemize}
\item[(\textbf{A1})]
$v^\epsilon,v:\Omega \times [0,T] \times \T \to \R^2$ and for every $t \in [0,T]$ the maps $v^\epsilon,v |_{\Omega \times [0,t]}:\Omega \times [0,t] \times \T \to \R^2$ are $\mathcal{F}_t \otimes \mathcal{B}_{[0,t]} \otimes \mathcal{B}_{\T}$ measurable, where $\mathcal{B}$ denotes the Borel sigma-field;
\item [(\textbf{A2})]
there exist a constant $C$ and a negligible set $\mathcal{N} \subset \Omega$ such that, for every $\omega \in \mathcal{N}^c$, $\epsilon > 0$ and $t \in [0,T]$: $\dvg v^\epsilon(\omega,t,\cdot) = \dvg v(\omega,t,\cdot) = 0$, and
\begin{gather*}
|v^\epsilon(\omega,t,x)|\leq C, \quad
|v^\epsilon(\omega,t,x) - v^\epsilon(\omega,t,y)| \leq C \gamma(|x-y|),
\\
|v(\omega,t,x)|\leq C, \quad
|v(\omega,t,x) - v(\omega,t,y)| \leq C \gamma(|x-y|),
\end{gather*}
for every $x,y \in \T$.
\end{itemize}
Also, we make the following assumption on the coefficients $(\varsigma_k)_{k \in \N}$:
\begin{itemize}
\item[(\textbf{A3})]
there exists $\ell \geq 1$ such that $\varsigma_k \in W^{\ell,\infty}(\T)$ with zero-mean for every $k \in \N$, and moreover
\begin{align*}
\sum_{k \in \N} \| \varsigma_k \|_{W^{\ell,\infty}(\T)} < \infty.
\end{align*}
\end{itemize}

Similarly to what has been done in the Introduction, given a stochastic flow of measure-preserving homeomorphisms $\phi$ we will use $\phi_t(x)$ as a notational shortcut for $\phi(\omega,\tilde{\omega},t,x)$, thus making implicit the dependence of the randomness variables $\omega,\tilde{\omega}$. The same convenction may be used for the fields $v,u,$ \emph{et cetera}.

The next result can be proved repeating the arguments contained in \cite{BrFlMa16} and \cite{FlPa21}.
\begin{prop} \label{prop:well_posed_small}
Assume (A1)-(A3). Then:
\begin{itemize}
\item
for every $\epsilon>0$ there exist a unique Lagrangian solution $\xi^\epsilon$ of \eqref{eq:euler_intro}, namely there exists a unique stochastic process $\xi^\epsilon : \Omega \times [0,T] \to L^\infty(\T)$ weakly progressively measurable with respect to $(\mathcal{F}_t)_{t \in [0,T]}$ such that the equation   
\begin{align*}
\psi^\epsilon_t(x) &= 
x
+
\int_0^t v^\epsilon_s(\psi^\epsilon_{s}(x)) ds
+
\int_0^t u^\epsilon_s(\psi^\epsilon_{s}(x)) ds,
\end{align*} 
with $u^\epsilon=K\ast \xi^\epsilon$, admits a unique inviscid stochastic flow of measure-preserving homeomorphisms $\psi^\epsilon$ as a solution, and moreover
\begin{align} \label{eq:xi_lagr}
\xi^\epsilon_t(\psi^\epsilon_t(x))
&=
\epsilon^{-1} \sum_{k \in \N}
\int_0^t e^{-\epsilon^{-1}(t-s)}
\varsigma_k(\psi^\epsilon_s(x)) dW^k_s;
\end{align}
\item 
for every $\epsilon>0$ there exists a unique stochastic flow of measure-preserving homeomorphisms $\phi^\epsilon$ solution of \eqref{eq:char_eps_intro}, with $u^\epsilon=K\ast \xi^\epsilon$;
\item
there exists a unique stochastic flow of measure-preserving homeomorphisms $\phi$ solution of \eqref{eq:char_intro}.
\end{itemize}
\end{prop}

\begin{rmk}
If $\nu=0$, then both $\phi^\epsilon$ and $\phi$ are inviscid stochastic flows of measure-preserving homeomorphisms, and actually $\phi^\epsilon=\psi^\epsilon$. The terminology is thus justified, since $\nu=0$ corresponds to null diffusivity/viscosity in the equations for the large-scale dynamics \eqref{eq:large_eps_intro} and \eqref{eq:large_intro}.
\end{rmk}

\begin{rmk}
Formula \eqref{eq:xi_lagr} above corresponds to the solution of \eqref{eq:euler_intro} with initial condition $\xi^\epsilon_0=0$, that we assume throughout this paper for the sake of simplicity.
More general initial conditions, as those considered in \cite{FlPa21}, can be taken into account by simply modifying \eqref{eq:xi_lagr} into
\begin{align*}
\xi^\epsilon_t(\psi^\epsilon_t(x))
&=
e^{-\epsilon^{-1}t}\xi^\epsilon_0(x)
+
\epsilon^{-1} \sum_{k \in \N}
\int_0^t e^{-\epsilon^{-1}(t-s)}
\varsigma_k(\psi^\epsilon_s(x)) dW^k_s.
\end{align*}
\end{rmk}

\subsubsection{Notion of solution to the large-scale dynamics}

By previous \autoref{prop:well_posed_small}, under assumption (A1)-(A3) we can use the Euler flow to represent the solutions of \eqref{eq:large_eps_intro} and \eqref{eq:large_intro}. To be more precise, our notion of solution is given exactly by those processes $\Xi^\epsilon$, $\Xi$ for which \eqref{eq:repr_large_eps} and \eqref{eq:repr_large} hold true, and it is inspired by the notion of generalized solution in \cite[Definition 2.2]{BrFl95}. 

\begin{definition} \label{def:sol}
Assume (A1)-(A3), $q^\epsilon,q \in L^1([0,T],L^\infty(\T))$ for every $\epsilon>0$ and $\Xi_0 \in L^\infty(\T)$.
Then:
\begin{itemize}
\item
for every $\epsilon>0$, a measurable map $\Xi^\epsilon : \Omega \times [0,T] \times \T \to \R$ is called \emph{generalized solution} to \eqref{eq:large_eps_intro} if it is compatible with $v^\epsilon$ and for every $t \in [0,T]$ it holds
\begin{align*} 
\Xi^\epsilon_t
&=
\tE{ \Xi_0 \circ (\phi^\epsilon_t)^{-1} + \int_0^t q^\epsilon_s \circ \phi^\epsilon_s \circ(\phi^\epsilon_t)^{-1}ds},
\end{align*}
as an equality in $L^\infty(\Omega \times \T)$, where $\phi^\epsilon$ is the unique stochastic flow of measure-preserving homeomorphisms solution of \eqref{eq:char_eps_intro};
\item
a measurable map $\Xi: \Omega \times [0,T] \times \T \to \R$ is called \emph{generalized solution} to \eqref{eq:large_intro} if it is compatible with $v$ and for every $t \in [0,T]$ it holds
\begin{align*} 
\Xi_t
&=
\tE{ \Xi_0 \circ (\phi_t)^{-1} + \int_0^t q_s \circ \phi_s \circ(\phi_t)^{-1}ds},
\end{align*}
as an equality in $L^\infty(\Omega \times \T)$, where $\phi$ is the unique stochastic flow of measure-preserving homeomorphisms solution of \eqref{eq:char_intro}.
\end{itemize}
\end{definition}

Some remark on the previous definition is appropriate.

First, notice that this notion of solution immediately implies existence and uniqueness in the case of passive large-scale dynamics: indeed, the compatibility condition is void, and $\Xi^\epsilon$ (resp. $\Xi$) depends only on the initial datum $\Xi_0$, the external sources $q^\epsilon$ (resp. $q$), and the characteristics $\phi^\epsilon$ (resp. $\phi$), the latter existing  and being unique by \autoref{prop:well_posed_small}.
For active dynamics this picture is not correct, since the compatibility condition between the external field and the large-scale variable is not encoded in the representation formula itself.
However, we will not investigate in this paper well-posedness for this notion of solution in full generality, but rather assume to have generalized solutions $\Xi^\epsilon$, $\Xi$ to work with. 

Secondly, it is worth of mention that every sufficiently smooth generalized solution of \eqref{eq:large_eps_intro} or \eqref{eq:large_intro} is also a classical solution, as can be proved following the lines of \cite[Theorem 2.2 and Proposition 2.7]{CoIy08}. On the other hand, our notion of generalized solution is weaker than the notion of $L^\infty$-weak solution contained in \cite{BrFlMa16}, that we recall now: 

\begin{definition}
Assume (A1)-(A3), $q^\epsilon,q \in L^1([0,T],L^\infty(\T))$ for every $\epsilon>0$ and $\Xi_0 \in L^\infty(\T)$. For $f,g:\T \to \R$, denote $\langle f, g \rangle \coloneqq \int_\T f(x)g(x)dx$. 
Then:
\begin{itemize}
\item
for every $\epsilon>0$, a stochastic process $\Xi^\epsilon : \Omega \times [0,T] \to L^\infty(\T)$ is called a \emph{$L^\infty$-weak solution} of \eqref{eq:large_eps_intro} if it is weakly progressively measurable with respect to $(\mathcal{F}_t)_{t \in [0,T]}$
and for every smooth test function $f \in C^\infty(\T)$ it holds $\PP$-a.s. for every $t \in [0,T]$:
\begin{align*}
\langle \Xi^\epsilon_t , f \rangle - \langle \Xi^\epsilon_0 , f \rangle
&=
\int_0^t 
\langle \Xi^\epsilon_s , (v^\epsilon_s + u^\epsilon_s) \cdot \nabla f \rangle ds
\\
&\quad
+
\int_0^t 
\langle \Xi^\epsilon_s , \nu \Delta f \rangle ds
+
\int_0^t 
\langle q^\epsilon_s , f \rangle ds;
\end{align*}
\item
a stochastic process $\Xi : \Omega \times [0,T] \to L^\infty(\T)$ is called a \emph{$L^\infty$-weak solution} of \eqref{eq:large_intro} if it is weakly progressively measurable with respect to $(\mathcal{F}_t)_{t \in [0,T]}$
and for every smooth test function $f \in C^\infty(\T)$ it holds $\PP$-a.s. for every $t \in [0,T]$:
\begin{align*}
\langle \Xi_t , f \rangle - \langle \Xi_0 , f \rangle
&=
\int_0^t 
\langle \Xi_s , v_s \cdot \nabla f \rangle ds
+
\sum_{k \in \N}
\int_0^t 
\langle \Xi_s , \sigma_k \cdot \nabla f \rangle \circ dW^k_s
\\
&\quad
+
\int_0^t 
\langle \Xi_s , \nu \Delta f \rangle ds
+
\int_0^t 
\langle q_s , f \rangle ds
.
\end{align*}
\end{itemize}
\end{definition}

\begin{prop}
Assume (A1)-(A3), $q^\epsilon,q \in L^1([0,T],L^\infty(\T))$ for every $\epsilon>0$ and $\Xi_0 \in L^\infty(\T)$.
Then every $L^\infty$-weak solution to \eqref{eq:large_eps_intro} is also a generalized solution to \eqref{eq:large_eps_intro}, and every $L^\infty$-weak solution to \eqref{eq:large_intro} is also a generalized solution to \eqref{eq:large_intro}. 
\end{prop}

\begin{proof}
The strategy of the proof is similar to \cite[Proposition 5.3]{BrFlMa16} and \cite[Theorem 20]{FlGuPr10}, and consists in taking the convolution of a $L^\infty$-weak solution with a smooth mollifier $\vartheta_\delta = \delta^{-2} \vartheta(\delta \cdot)$, $\delta>0$, and then taking the limit for $\delta \to 0$.

Let $\Xi^\epsilon$ be a $L^\infty$-weak solution of \eqref{eq:large_eps_intro} and $\Xi$ be a $L^\infty$-weak solution of \eqref{eq:large_intro}, in the sense of the previous definition. Using $f = \vartheta_\delta(y-\cdot)$ as a test function, $y \in \T$, and denoting $\Xi^\epsilon_\delta \coloneqq \vartheta_\delta \ast \Xi^\epsilon$, $\Xi_\delta \coloneqq \vartheta_\delta \ast \Xi$ we get (omitting the parameter $\omega$)
\begin{align*}
\Xi^\epsilon_\delta(t,y) - \Xi^\epsilon_\delta(0,y)
&=
\int_0^t \int_\T
\Xi^\epsilon(s,x) (v^\epsilon(s,x) + u^\epsilon(s,x)) \cdot \nabla_x \vartheta_\delta (y-x) dx ds
\\
&\quad
+
\nu \int_0^t \int_\T
\Xi^\epsilon(s,x) \Delta_x \vartheta_\delta (y-x) dx ds
\\
&\quad
+
\int_0^t \int_\T
q^\epsilon(s,x) \vartheta_\delta (y-x) dx ds,
\end{align*}
and
\begin{align*}
\Xi_\delta(t,y) - \Xi_\delta(0,y)
&=
\int_0^t \int_\T
\Xi(s,x) v(s,x) \cdot \nabla_x \vartheta_\delta (y-x) dx ds
\\
&\quad
+ \sum_{k \in \N}
\int_0^t \int_\T
\Xi(s,x) \sigma_k(x) \cdot \nabla_x \vartheta_\delta (y-x) dx \circ dW^k_s
\\
&\quad
+
\nu \int_0^t \int_\T
\Xi(s,x) \Delta_x \vartheta_\delta (y-x) dx ds
\\
&\quad
+
\int_0^t \int_\T
q(s,x) \vartheta_\delta (y-x) dx ds.
\end{align*}

Since $\Xi^\epsilon_\delta$, $\Xi_\delta$ are smooth functions in the variable $y$, we can write the equivalent expressions in differential notation
\begin{align*}
d\Xi^\epsilon_\delta(t,y)
&+
\nabla\Xi^\epsilon_\delta(t,y) \cdot(v^\epsilon(t,y) + u^\epsilon(t,y)) dt 
\\
&=
\int_\T
\Xi^\epsilon(t,x) (v^\epsilon(t,x) + u^\epsilon(t,x)) \cdot \nabla_x \vartheta_\delta (y-x) dx dt
\\
&\quad
+
\nu \int_\T
\Xi^\epsilon(t,x) \Delta_x \vartheta_\delta (y-x) dx dt
+
\int_\T
q^\epsilon(t,x) \vartheta_\delta (y-x) dx  dt
\\
&\quad
+
\nabla \Xi^\epsilon_\delta(t,y) \cdot (v^\epsilon(t,y) + u^\epsilon(t,y)) dt,
\end{align*}
and
\begin{align*}
d\Xi_\delta(t,y)
&+
\nabla\Xi_\delta(t,y) \cdot v(t,y) dt
+
\sum_{k \in \N}
\nabla\Xi_\delta(t,y) \cdot \sigma_k(y) \circ dW^k_t
\\
&=
\int_\T
\Xi(t,x) v(t,x) \cdot \nabla_x \vartheta_\delta (y-x) dx dt
\\
&\quad
+ \sum_{k \in \N}
\int_\T
\Xi(t,x) \sigma_k(x) \cdot \nabla_x \vartheta_\delta (y-x) dx \circ dW^k_t
\\
&\quad
+
\nu\int_\T
\Xi(t,x) \Delta_x \vartheta_\delta (y-x) dx dt
+
\int_\T
q(t,x) \vartheta_\delta (y-x) dx dt
\\
&\quad+
\nabla\Xi_\delta(t,y) \cdot v(t,y) dt
+
\sum_{k \in \N}
\nabla\Xi_\delta(t,y) \cdot \sigma_k(y) \circ dW^k_t.
\end{align*}

Notice that the following formulas for the gradient of the convolution hold true: $\nabla \Xi^\epsilon_\delta(t,y) = -\int_\T \Xi^\epsilon(t,x) \nabla_x \vartheta_\delta(y-x)$, and $\nabla \Xi_\delta(t,y) = -\int_\T \Xi(t,x) \nabla_x \vartheta_\delta(y-x)$; also, $\Delta_x \vartheta_\delta (y-x) = \Delta_y \vartheta_\delta (y-x)$.
Substituting into the previous expressions, we get
\begin{align*}
d\Xi^\epsilon_\delta(t,y)
&+
\nabla\Xi^\epsilon_\delta(t,y) \cdot(v^\epsilon(t,y) + u^\epsilon(t,y)) dt 
\\
&=
\left[ - \vartheta_\delta \ast \left( \nabla \Xi^\epsilon_t \cdot \left( v^\epsilon_t + u^\epsilon_t \right) \right) + \left( v^\epsilon_t + u^\epsilon_t \right) \cdot \left( \vartheta_\delta \ast \nabla \Xi^\epsilon_t \right)
\right] (y) dt
\\
&\quad
+
\nu \Delta \Xi^\epsilon_\delta(t,y) dt
+
q^\epsilon_\delta(t,y) dt
\\
&=
R_\delta 
\left[ v^\epsilon_t + u^\epsilon_t , \Xi^\epsilon_t \right] (y) dt
+
\nu \Delta \Xi^\epsilon_\delta(t,y) dt
+
q^\epsilon_\delta(t,y) dt,
\end{align*}
and
\begin{align*}
d\Xi_\delta(t,y)
&+
\nabla\Xi_\delta(t,y) \cdot v(t,y) dt
+
\sum_{k \in \N}
\nabla\Xi_\delta(t,y) \cdot \sigma_k(y) \circ dW^k_t
\\
&=
\left[ - \vartheta_\delta \ast \left( \nabla \Xi_t \cdot  v_t\right) +  v_t  \cdot \left( \vartheta_\delta \ast \nabla \Xi_t\right)\right] (y) dt
\\
&\quad
+ \sum_{k \in \N}
\left[ - \vartheta_\delta \ast \left( \nabla \Xi_t \cdot \sigma_k \right) +  \sigma_k  \cdot \left( \vartheta_\delta \ast \nabla \Xi_t\right)\right] (y) \circ dW^k_t
\\
&\quad
+
\nu 
\Delta \Xi_\delta(t,y) dt
+
q_\delta(t,y) dt
\\
&=
R_\delta 
\left[ v_t , \Xi_t \right] (y) dt
+ 
\sum_{k \in \N}
R_\delta 
\left[ \sigma_k , \Xi_t \right] (y) \circ dW^k_t
\\
&\quad
+
\nu 
\Delta \Xi_\delta(t,y) dt
+
q_\delta(t,y) dt,
\end{align*}
where we have defined $q^\epsilon_\delta \coloneqq \vartheta_\delta \ast q^\epsilon$, $q_\delta \coloneqq \vartheta_\delta \ast q$ and the commutator
\begin{align*}
R_\delta 
\left[ v , \Xi \right]
\coloneqq
- \vartheta_\delta \ast \left( \nabla \Xi \cdot v \right) +  v  \cdot \left( \vartheta_\delta \ast \nabla \Xi \right).
\end{align*}

We have obtained differential equations for the spatially smooth processes $\Xi^\epsilon_\delta$ and $\Xi_\delta$.
Applying the backwards It\=o Formula to the processes $s \mapsto \Xi^\epsilon_\delta(s,\phi^\epsilon_s ((\phi^\epsilon_t)^{-1}(y)))$ and $s \mapsto \Xi_\delta(s,\phi_s ((\phi_t)^{-1}(y)))$, for fixed $t \in [0,T]$, and taking the expectation with respect to $\tilde{\PP}$, we obtain that the process $\Xi^\epsilon_\delta$ is given by
\begin{align} \label{eq:Xi_esp_delta}
\Xi^\epsilon_\delta(t,y)
&=
\tE{\Xi^\epsilon_\delta(0,(\phi^\epsilon_t)^{-1}(y)) 
+
\int_0^t q^\epsilon_\delta(s,\phi^\epsilon_s((\phi^\epsilon_t)^{-1}(y)))
ds
}
\\
&\quad+  \nonumber
\tE{ \int_0^t R_\delta 
\left[ v^\epsilon_s + u^\epsilon_s , \Xi^\epsilon_s \right] (\phi^\epsilon_s((\phi^\epsilon_t)^{-1}(y))) ds},
\end{align}
whereas the process $\Xi_\delta$ is given by
\begin{align} \label{eq:Xi_delta}
\Xi_\delta(t,y)
&=
\tE{\Xi_\delta(0,(\phi_t)^{-1}(y)) 
+
\int_0^t q_\delta(s,\phi_s((\phi_t)^{-1}(y)))
ds
}
\\
&\quad+  \nonumber
\tE{ \int_0^t R_\delta 
\left[ v_s , \Xi_s \right] (\phi_s((\phi_t)^{-1}(y))) ds}
\\
&\quad+
\sum_{k\in \N}
\tE{ \int_0^t R_\delta 
\left[ \sigma_k , \Xi_s \right] (\phi_s((\phi_t)^{-1}(y))) \circ dW^k_s} \nonumber
\end{align}

Let us focus on \eqref{eq:Xi_esp_delta}. 
By well-known properties of mollifiers, for every fixed $\omega \in \Omega $ and $t \in [0,T]$, the right-hand side $\Xi^\epsilon_\delta(\omega,t,\cdot) \to \Xi^\epsilon(\omega,t,\cdot)$ in $L^1(\T)$ as $\delta \to 0$. 
Concerning the left-hand side, a commutator lemma \cite[Lemma 17]{FlGuPr10} yields for every fixed $\epsilon>0$
\begin{align*}
\lim_{\delta \to 0}
\int_\T
\left|\tE{ \int_0^t R_\delta 
\left[ v^\epsilon_s + u^\epsilon_s , \Xi^\epsilon_s \right] (\phi^\epsilon_s((\phi^\epsilon_t)^{-1}(y))) ds} \right| dy
= 0,
\end{align*}
and by well-known properties of mollifiers and Lebesgue dominated convergence Theorem we can prove the convergence
\begin{align*}
&\tE{\Xi^\epsilon_\delta(0,(\phi^\epsilon_t)^{-1}) 
+
\int_0^t q^\epsilon_\delta(s,\phi^\epsilon_s((\phi^\epsilon_t)^{-1}))
ds
}
\\
&\quad+  \nonumber
\tE{ \int_0^t R_\delta 
\left[ v^\epsilon_s + u^\epsilon_s , \Xi^\epsilon_s \right] (\phi^\epsilon_s((\phi^\epsilon_t)^{-1})) ds}
\\
&\to \tE{\Xi^\epsilon(0,(\phi^\epsilon_t)^{-1}) 
+
\int_0^t q^\epsilon(s,\phi^\epsilon_s((\phi^\epsilon_t)^{-1}))
ds
}
\end{align*}
in $L^1(\T)$ as $\delta \to 0$, for almost every $\omega \in \Omega$ and $t \in [0,T]$.
Therefore, by \eqref{eq:Xi_delta} we have and the uniqueness of the $L^1(\T)$ limit, for almost every $\omega \in \Omega$, $t \in [0,T]$ and $y \in \T$:
\begin{align*}
\Xi^\epsilon(t,y)
=
\tE{\Xi^\epsilon(0,(\phi^\epsilon_t)^{-1}(y)) 
+
\int_0^t q^\epsilon(s,\phi^\epsilon_s((\phi^\epsilon_t)^{-1}(y)))
ds
},
\end{align*}
that is exactly the desired representation formula \eqref{eq:repr_large_eps}. The argument for \eqref{eq:Xi_delta} is similar, with only a little complication due to the stochastic integral, and we leave it to the reader.
\end{proof}

As a final remark, since we have seen that the notion of generalized solution is weaker than the notion of $L^\infty$-weak solution, our results are indeed very general: they can be applied at least to every $L^\infty$-weak solution.

\subsection{Statement of main results}
\subsubsection{Convergence of characteristics}

Denote $|x-y|$ the geodesic distance on the flat two dimensional torus between points  $x,y \in \T$. 
To keep the notation simple, we define the following quantity associated with a measurable map $\phi:\T \to \T$:
\begin{align*}
\|\phi\|_{L^1(\T,\T)}
\coloneqq
\int_\T
|\phi(x)| dx.
\end{align*}
Notice that $\| \cdot \|_{L^1(\T,\T)}$ is not a norm on the space of measurable maps $\phi:\T\to\T$, in particular it is not positively homogeneous. However, $\| \cdot \|_{L^1(\T,\T)}$ induces a distance on the space $C(\T,\T)$ of continuous maps $\phi:\T \to \T$.
Similarly, we define $\| \cdot \|_{L^\infty(\T,\T)}$ as
\begin{align*}
\| \phi \|_{L^\infty(\T,\T)}
\coloneqq
\mbox{ess}\sup_{x \in \T}
|\phi(x)|.
\end{align*}

In order to prove convergence of characteristics $\phi^\epsilon \to \phi$, it is clear that one needs some sort of control for the difference $v^\epsilon - v$. Therefore, we assume: 
\begin{itemize}
\item[(\textbf{A4})]
there exist a constant $C$ and a negligible set $\mathcal{N} \subset \Omega$ such that for every $\omega \in \mathcal{N}^c$, $\epsilon > 0$ and $t \in [0,T]$:
\begin{align*}
\| v^\epsilon(\omega,t,\cdot) - v(\omega,t,\cdot) \|_{L^1(\T,\R^2)}
&\leq C \gamma\left( \tE{\| \phi^\epsilon_t - \phi_t \|_{L^1(\T,\T)}} \right)
\\
&\quad+
C \int_0^t \gamma\left( \tE{\| \phi^\epsilon_s - \phi_s \|_{L^1(\T,\T)}} \right) ds
+
c_\epsilon,
\end{align*}
where $c_\epsilon \in \R$ is infinitesimal as $\epsilon\to 0$, $\phi^\epsilon_t=\phi^\epsilon(\omega,\tilde{\omega},t,\cdot)$ is the unique solution of \eqref{eq:char_eps_intro}, and $\phi_t=\phi(\omega,\tilde{\omega},t,\cdot)$ is the unique solution of \eqref{eq:char_intro}. 
\end{itemize}

A little less clear, at this point, is our next assumption on the coefficients $(\varsigma_k)_{k \in \N}$:
\begin{itemize}
\item[(\textbf{A5})]
for every $x \in \T$ it holds
\begin{align*}
\sum_{k \in \N} ((K\ast \varsigma_k) \cdot \nabla \varsigma_k)(x) = 0.
\end{align*}
\end{itemize}
The motivations for assuming (A5) will become evident during the proof of \autoref{prop:sup_z} in \autoref{sec:tech}. 

We are ready to state our first main result: 
\begin{thm} \label{thm:char}
Assume (A1)-(A5).
Let $\hE{\cdot} \coloneqq \E{\tE{\cdot}}$ denote the expectation on $\hat{\Omega} \coloneqq \Omega \times \tilde{\Omega}$ with respect to the probability measure $\hat{\PP} \coloneqq \PP \otimes \tilde{\PP}$.
Then
\begin{align*}
\sup_{t \in [0,T]}  \hE{ \| \phi_t^\epsilon-\phi_t \|_{L^1(\T,\T)}}
\to 0
\quad \mbox{ as } \epsilon \to 0.
\end{align*}
\end{thm}

\subsubsection{Convergence of large-scale dynamics}

Let $q^\epsilon,q : [0,T] \times \T \to \R$ be such that:
\begin{itemize}
\item[(\textbf{A6})]
there exists a constant $C$ such that for every $\epsilon>0$ it holds $q^\epsilon,q \in$\\ $L^1([0,T],L^\infty(\T))$ and
\begin{align*}
\int_0^T \| q^\epsilon_s \|_{L^\infty(\T)} ds \leq C,
\qquad
\int_0^T \| q_s \|_{L^\infty(\T)} ds \leq C;
\end{align*}
\item[(\textbf{A7})]
$q^\epsilon-q$ converges to zero in $L^1([0,T],L^\infty(\T))$.
\end{itemize}
Our second main result is the following:
\begin{thm} \label{thm:conv_large}
Assume (A1)-(A7) and $\Xi_0\in L^\infty(\T)$.
Then the solution $\Xi^\epsilon$ of \eqref{eq:large_eps_intro} converges towards the solution $\Xi$ of \eqref{eq:large_intro} in the following sense: for every $f \in L^1(\T)$
\begin{align*}
\E{\left|\int_\T \Xi^\epsilon_t(x)  f(x)dx
-\int_\T \Xi_t(x)  f(x)dx \right|} \to 0
\quad \mbox{ as } \epsilon \to 0,
\end{align*}
for every fixed $t\in [0,T]$ and in $L^p([0,T])$ for every finite $p$.
Moreover, if $q \in L^1([0,T],Lip(\T))$ then the previous convergence holds uniformly in $t \in [0,T]$ and $f \in Lip(\T)$ with Lipschitz constant $[f]_{Lip(\T)} \leq 1$ and $\|f\|_{L^\infty(\T)} \leq 1$.
\end{thm}

\section{Technical results} \label{sec:tech}

In this section and after in the paper, the symbol $\lesssim$ will indicate inequality up to a unimportant multiplicative constant $C$ not depending of $\epsilon$.

\subsection{Linearized dynamics}
For $\epsilon>0$, denote $\theta^\epsilon$ the solution of the linear problem
\begin{align*}
d\theta^\epsilon_t=-\epsilon^{-1} \theta^\epsilon_t dt + \epsilon^{-1} \sum_{k \in\N} \varsigma_k dW^k_t,
\end{align*}
with initial condition $\theta^\epsilon|_{t=0}=0$.
The process $\theta^\epsilon$ is explicitly given by the formula $\theta^\epsilon_t = 
\sum_{k \in \N} \varsigma_k \eta^{\epsilon,k}_t$, where
\begin{align*}
\eta^{\epsilon,k}_t \coloneqq \epsilon^{-1} \int_0^t e^{-\epsilon^{-1}(t-s)} dW^k_s,
\quad
k \in \N,
\end{align*} 
is the so called Ornstein-Uhlenbeck process with null initial condition. By \cite[Theorem 2.2]{JiZh20}, for every fixed $p \geq 1$ it holds uniformly in $k \in \N$
\begin{align} \label{eq:OU}
\E{\sup_{t \in [0,T]} |\eta^{\epsilon,k}_t|^p } \lesssim \epsilon^{-p/2} \log^{p/2}(1+\epsilon^{-1}), 
\end{align}
and therefore by assumption (A3)
\begin{align} \label{eq:est_theta}
\E{\sup_{t \in [0,T]} \|\theta^\epsilon_t\|^p_{W^{1,\infty}(\T)} } \lesssim \epsilon^{-p/2} \log^{p/2}(1+\epsilon^{-1}). 
\end{align}

The difference $\zeta^\epsilon \coloneqq \xi^{\epsilon} - \theta^\epsilon$ between the small-scale vorticity $\xi^\epsilon$ and $\theta^\epsilon$ solves the equation
\begin{align*}
d\zeta^\epsilon_t + (v^\epsilon_t+u^\epsilon_t)\cdot\nabla \zeta^\epsilon_t dt = -\epsilon^{-1} \zeta^\epsilon_t dt - (v^\epsilon_t+u^\epsilon_t)\cdot\nabla \theta^\epsilon_t dt
\end{align*}
with initial condition $\zeta^\epsilon_0=0$, whose solution satisfies
\begin{align} \label{eq:zeta_psi}
\zeta^\epsilon_t(\psi^\epsilon_t(x)) 
&=
-\int_0^t
e^{-\epsilon^{-1}(t-s)}
((v^\epsilon_s + u^\epsilon_s) \cdot\nabla \theta^\epsilon_s)(\psi^\epsilon_s(x)) ds. 
\end{align} 
In the following, for $t \in [0,T]$ and $x \in \T$ we denote $z^\epsilon_t(x)  = (K \ast \zeta^\epsilon_t)(x)$.

\subsection{Main technical results}
We are going to prove two main technical results, needed for the proof of \autoref{thm:char}. Since our strategy consists in replicating the proof of \cite[Proposition 4.1]{FlPa21}, the first result we need is the following:
\begin{prop} \label{prop:old}
Assume (A1)-(A3). 
Then the following inequality holds: 
\begin{align*}
\hE{ \sup_{t \in [0,T]}\left\|
\sum_{k \in \N}
\int_0^t
\sigma_k(\phi^\epsilon_s(\cdot)) \eta^{\epsilon,k}_s ds
-
\sum_{k \in \N}
\int_0^t
\sigma_k(\phi^\epsilon_s(\cdot)) \circ dW^k_s
\right\|_{L^1(\T,\R^2)}}
\\
\lesssim
\epsilon^{1/42} \log^{47/42}(1+\epsilon).
\end{align*}
\end{prop}
In \cite[Section 4]{FlPa21} a similar estimate was proven along the way, using a considerable amount of auxiliary lemmas and computations. In view of this, here we refrain from going again into full detail, and the proof of \autoref{prop:old} will only be sketched.

On the other hand, the nonlinear term in \eqref{eq:euler_intro} produces a new term in the equation of characteristcs, that was absent in \cite{FlPa21}.
Although the final results is not affected by this new term, it is not trivial to actually prove so.
We need the following: 
\begin{prop} \label{prop:sup_z}
Assume (A1)-(A5). Then:
\begin{align*}
\hE{\sup_{t \in [0,T]}\left\|\int_0^t
z^\epsilon_s(\phi^\epsilon_s(\cdot)) ds\right\|_{L^\infty(\T,\R^2)}}
\lesssim
\epsilon^{1/12} \log^{11/12}(1+\epsilon^{-1}).
\end{align*}
\end{prop}
This constitutes the main novelty with respect to \cite{FlPa21}.
The proof of \autoref{prop:sup_z} relies strongly on assumption (A5) and the following It\=o Formulas, yielding for every fixed $t\in [0,T]$ and $k,h \in \N$:
\begin{align*}
\eta^{\epsilon,k}_t \eta^{\epsilon,h}_t
&= 
e^{-\epsilon^{-1} t }\eta^{\epsilon,k}_0 \eta^{\epsilon,h}_0
-
\epsilon^{-1} \int_0^t e^{-\epsilon^{-1}(t-s)}
\eta^{\epsilon,k}_s \eta^{\epsilon,h}_s ds
\\
&\quad
+\epsilon^{-1} \int_0^t e^{-\epsilon^{-1}(t-s)} \eta^{\epsilon,k}_s dW^h_s
+\epsilon^{-1} \int_0^t e^{-\epsilon^{-1}(t-s)} \eta^{\epsilon,h}_s dW^k_s
\\
&\quad
+\delta_{k,h} \frac{\epsilon^{-2}}{2} \int_0^t e^{-\epsilon^{-1}(t-s)} ds,
\\
\eta^{\epsilon,k}_t \eta^{\epsilon,h}_t
&=
\eta^{\epsilon,k}_0 \eta^{\epsilon,h}_0
-
2\epsilon^{-1} \int_0^t 
\eta^{\epsilon,k}_s \eta^{\epsilon,h}_s ds
\\
&\quad
+\epsilon^{-1} \int_0^t  \eta^{\epsilon,k}_s dW^h_s
+\epsilon^{-1} \int_0^t  \eta^{\epsilon,h}_s dW^k_s
+\delta_{k,h} \frac{\epsilon^{-2} t}{2} ,
\end{align*}
with $\delta_{k,h}$ being the Kronecker delta function, allowing to control the time integral of quadratics $\eta^{\epsilon,k}_s \eta^{\epsilon,h}_s$.

\subsection{Proof of \autoref{prop:old}}
 
In this paragraph we recall the argument contained in \cite{FlPa21}. 
Roughly speaking, \autoref{prop:old} is a sort of Wong-Zakai result for the Ornstein-Uhlenbeck process $\eta^{\epsilon,k}$ converging to a white-in-time noise, that is the formal time derivative of the Wiener process $W^k$.
We need to exploit a discretization of \eqref{eq:char_eps_intro} to show the closeness, in a certain sense to be specified, between the Stratonovich-to-It\=o corrector $c:\mathbb{T}^2 \to \mathbb{R}^2$, given by:
\begin{align*}
c(x)=\frac12 \sum_{k \in \mathbb{N}} \nabla\sigma_k(x) \cdot \sigma_k(x),
\quad
x \in \mathbb{T}^2,
\end{align*} 
coming from the stochastic integral, and the iterated time integral of the Ornstein-Uhlenbeck process.

In order to discretize the problem, for every $\epsilon>0$ take a mesh $\delta>0$ such that $T/\delta$ is an integer.
For any $n=0,\dots,T/\delta-1$ and fixed $x \in \T$, consider the following decomposition:
\begin{align*}
\sum_{k \in \mathbb{N}} \int_{n\delta}^{(n+1)\delta} \sigma_k(\phi^\epsilon_s(x)) \eta^{\epsilon,k}_s ds 
&=
\sum_{k \in \mathbb{N}}\int_{n\delta}^{(n+1)\delta}
\left(
\int_{n\delta}^{s} 
\nabla\sigma_k(\phi^\epsilon_r(x)) \cdot
v^\epsilon_r(\phi^\epsilon_r(x)) dr
\right)\eta^{\epsilon,k}_s ds 
\\
&\quad+
\sum_{k \in \mathbb{N}}\int_{n\delta}^{(n+1)\delta}
\left(
\int_{n\delta}^{s} 
\nabla\sigma_k(\phi^\epsilon_r(x)) \cdot
z^\epsilon_r(\phi^\epsilon_r(x)) dr
\right)\eta^{\epsilon,k}_s ds 
\\
&\quad+
\sum_{k,h \in \mathbb{N}}\int_{n\delta}^{(n+1)\delta}
\left(
\int_{n\delta}^{s} 
\nabla\sigma_k(\phi^\epsilon_r(x)) \cdot
\sigma_h(\phi^\epsilon_r(x)) \eta^{\epsilon,h}_r dr
\right)\eta^{\epsilon,k}_s ds 
\\
&\quad+
\sum_{k \in \mathbb{N}}\int_{n\delta}^{(n+1)\delta}
\left(
\int_{n\delta}^{s} 
\nabla\sigma_k(\phi^\epsilon_r(x)) \cdot
\sqrt{2\nu} dw_r
\right)\eta^{\epsilon,k}_s ds 
\\
&\quad+
\sum_{k \in \mathbb{N}}\int_{n\delta}^{(n+1)\delta}  \sigma_k(\phi^\epsilon_{n\delta}(x))
dW^k_s 
\\
&\quad-
\sum_{k \in \mathbb{N}} \int_{n\delta}^{(n+1)\delta} \sigma_k(\phi^\epsilon_{n\delta}(x))
\epsilon d\eta^{\epsilon,k}_s  \\
\eqqcolon&\, I^\epsilon_1(n) + I^\epsilon_2(n) + I^\epsilon_3(n) + I^\epsilon_4(n) + I^\epsilon_5(n) + I^\epsilon_6(n).
\end{align*}

Regarding the Stratonovich integral, we can rewrite:
\begin{align*}
\sum_{k\in \mathbb{N}} \int_{n\delta}^{(n+1)\delta}\sigma_k(\phi^\epsilon_s(x)) \circ dW^k_s 
=
&\sum_{k\in \mathbb{N}}\int_{n\delta}^{(n+1)\delta}
 \left(\sigma_k(\phi^\epsilon_s(x))-\sigma_k(\phi^\epsilon_{n\delta}(x))\right) dW^k_s \\
&+
\sum_{k\in \mathbb{N}}\int_{n\delta}^{(n+1)\delta}
 \sigma_k(\phi^\epsilon_{n\delta}(x)) 
dW^k_s \\
&+
\int_{n\delta}^{(n+1)\delta}
\left( c(\phi^\epsilon_s(x))- c(\phi^\epsilon_{n\delta}(x))  \right) ds \\
&+
\int_{n\delta}^{(n+1)\delta}
c(\phi^\epsilon_{n\delta}(x)) ds \\
\eqqcolon&\, J^\epsilon_1(n) + J^\epsilon_2(n) + J^\epsilon_3(n) + J^\epsilon_4(n).
\end{align*}

The ingredients for the proof of \autoref{prop:old} are:
\begin{itemize}
\item
a good estimate on $\hE{\sup_{t \in [0,T]} |z^\epsilon_t(\phi^\epsilon_t(x))|}$ (cfr. \autoref{lem:zeta_log}), needed to control $I^\epsilon_2(n)$;
\item
a good estimate on $\hE{\sup_{\tau \leq \delta} |\phi^\epsilon_{\tau + n\delta}(x)-\phi^\epsilon_{n\delta}(x)|}$ (cfr. \autoref{lem:phi_eps_incr}),
needed to approximate $I^\epsilon_3(n)$ with 
\begin{align} \label{eq:I_3approx}
\sum_{k,h \in \mathbb{N}}
\nabla\sigma_k(\phi_{n\delta}(x)) \cdot
\sigma_h(\phi_{n\delta}(x))
\int_{n\delta}^{(n+1)\delta}
 \left(
\int_{n\delta}^{s} 
\eta^{\epsilon,h}_r dr
\right)\eta^{\epsilon,k}_s ds ;
\end{align}
\item
a better estimate on $\hE{|\phi^\epsilon_{(n+1)\delta}(x)-\phi^\epsilon_{n\delta}(x)|}$ (cfr. \autoref{lem:phi_eps_incr_bis}), needed to control $I^\epsilon_6(n)$ with a discrete integration by parts.
\end{itemize}

Notice that $I^\epsilon_5(n) = J^\epsilon_2(n)$, and the expression in \eqref{eq:I_3approx}, that approximates $I^\epsilon_3(n)$, must be compensated by subtracting $J^\epsilon_4(n)$.

\begin{lem} \label{lem:zeta_log}
Assume (A1)-(A3). Then for every fixed $p\geq 1$ it holds
\begin{align*}
\E{\sup_{t \in [0,T]} \|\zeta^\epsilon_t\|^p_{L^\infty(\T)} } \lesssim \log^p(1+\epsilon^{-1}). 
\end{align*}
In particular, since $z^\epsilon_t = K \ast \zeta^\epsilon_t$ we alse have
\begin{align*}
\E{\sup_{t \in [0,T]} \|z^\epsilon_t\|^p_{L^\infty(\T)} } \lesssim \log^p(1+\epsilon^{-1}). 
\end{align*}
\end{lem}
\begin{proof}
We prove in the first place the weaker estimate:
\begin{align} \label{eq:est_zeta_loose}
\E{\sup_{t \in [0,T]} \|\zeta^\epsilon_t\|^p_{L^\infty(\T)} } \lesssim \epsilon^{-p}. 
\end{align}
Since $\theta^\epsilon$ satisfies the bound above by \eqref{eq:est_theta}, it suffices to prove it for $\xi^\epsilon$. Denote $M^\epsilon_t(x) = \sum_{k \in \N}
\int_0^t \varsigma_k(\psi^\epsilon_s(x)) dW^k_s$. Since for every $s,t \in [0,T]$
\begin{align*}
\E{\| M^\epsilon_t - M^\epsilon_s \|^4_{L^\infty(\T)}}
\lesssim
\left( \sum_{k \in \N} \| \varsigma_k\|_{L^\infty(\T)}^2 \right)^2 (t-s)^2,
\end{align*}
by (A3) and Kolmogorov continuity Theorem the process $M^\epsilon:\Omega \times [0,T] \to L^\infty(\T)$ has a modification $\tilde{M}^\epsilon$ that is $\alpha$-H\"older continuous for every $\alpha<1/4$, with $\alpha$-H\"older constant $K_{\epsilon,\alpha}$ bounded in $L^p(\Omega)$ for every $p<\infty$ uniformly in $\epsilon$. 
Since $M^\epsilon$ has continuous trajectories, $M^\epsilon_t=\tilde{M}^\epsilon_t$ a.s. as random variables in $L^\infty(\T)$ and 
\begin{align*}
\xi^\epsilon_t(\psi^\epsilon_t(x))
&=
\epsilon^{-1} 
\int_0^t e^{-\epsilon^{-1}(t-s)} dM^\epsilon_s(x)
\\
&=
\epsilon^{-1} 
\int_0^t e^{-\epsilon^{-1}(t-s)} d(M^\epsilon_s(x)-M^\epsilon_t(x))
\\
&=
\epsilon^{-1} \left[ e^{-\epsilon^{-1}(t-s)} 
(M^\epsilon_s(x)-M^\epsilon_t(x)) \right]^{s=t}_{s=0}
\\
&\quad
- \epsilon^{-2} 
\int_0^t e^{-\epsilon^{-1}(t-s)} 
(M^\epsilon_s(x)-M^\epsilon_t(x)) ds.
\end{align*}
Clearly $\|\xi^\epsilon_t\|_{L^\infty(\T)} = \| \xi^\epsilon_t \circ \psi^\epsilon_t \|_{L^\infty(\T)}$, and therefore
\begin{align*}
\| \xi^\epsilon_t \|_{L^\infty(\T)}
&\leq
\epsilon^{-1} e^{-\epsilon^{-1}t} \|M^\epsilon_t\|_{L^\infty(\T)}
+\epsilon^{-1} K_{\epsilon,\alpha},
\end{align*}
and \eqref{eq:est_zeta_loose} follows.

Recalling \eqref{eq:zeta_psi}, the following inequality holds
\begin{align} \label{eq:zeta_iterative}
\|\zeta^\epsilon_t\|_{L^\infty(\T)}
\leq
\int_0^t
e^{-\epsilon^{-1}(t-s)}
\|(v^\epsilon_s + u^\epsilon_s) \cdot\nabla \theta^\epsilon_s \|_{L^\infty(\T)} ds.
\end{align}
Using assumption (A2) and $u^\epsilon_s = K \ast \zeta^\epsilon_s + K \ast \theta^\epsilon_s$ we get
\begin{align*}
\|(v^\epsilon_s + u^\epsilon_s) \cdot\nabla \theta^\epsilon_s \|_{L^\infty(\T)}
&\lesssim
\left(
1 + \| \zeta^\epsilon_s \|_{L^\infty(\T)} + \| \theta^\epsilon_s \|_{L^\infty(\T)} \right)
\| \nabla \theta^\epsilon_s \|_{L^\infty(\T)},
\end{align*}
that can be plugged back into \eqref{eq:zeta_iterative} to produce the recursive estimate
\begin{align*}
\|\zeta^\epsilon_t\|_{L^\infty(\T)}
&\lesssim
\int_0^t e^{-\epsilon^{-1}(t-s)}
\left(
1 + \| \theta^\epsilon_s \|_{L^\infty(\T)} \right)
\| \nabla \theta^\epsilon_s \|_{L^\infty(\T)} ds
\\
&\quad+
\int_0^t e^{-\epsilon^{-1}(t-s)}
\|\zeta^\epsilon_s \|_{L^\infty(\T)}
\| \nabla \theta^\epsilon_s \|_{L^\infty(\T)} ds
\\
&\lesssim
\epsilon 
\left(
\sup_{s \in [0,T]}\| \theta^\epsilon_s \|_{L^\infty(\T)}
+
\sup_{s \in [0,T]}\| \zeta^\epsilon_s \|_{L^\infty(\T)}
\right)
\sup_{s \in [0,T]}\| \nabla \theta^\epsilon_s \|_{L^\infty(\T)}.
\end{align*}
By H\"older inequality and \eqref{eq:est_zeta_loose} we deduce from the previous inequality
\begin{align*}
\E{ \sup_{t \in [0,T]} \|\zeta^\epsilon_t\|^p_{L^\infty(\T)}}
\lesssim
\log^p(1+\epsilon^{-1}) 
+
\epsilon^{-p/2} \log^{p/2}(1+\epsilon^{-1}),
\end{align*}
improving the bound \eqref{eq:est_zeta_loose} itself. Iterating the same argument one more time we obtain the desired estimate.
\end{proof}

\begin{lem} \label{lem:phi_eps_incr}
Assume (A1)-(A3). Then for every fixed $p\geq 1$ and $\alpha \in (0,1/2)$
\begin{equation*}
\hE{\sup_{\substack{t+\tau \leq T \\ \tau \leq \delta}}
\| \phi^\epsilon_{t+\tau} - \phi^\epsilon_t \|^p_{L^{\infty}(\mathbb{T}^2,\mathbb{T}^2)}} 
\lesssim
\delta^p \epsilon^{-p/2} \log^{p/2}(1+\epsilon^{-1}) + \delta^{p\alpha}.
\end{equation*}
\end{lem}
\begin{proof}
The increment $\phi^\epsilon_{t+\tau}(x) - \phi^\epsilon_{t}(x)$ can be written as
\begin{align*}
\phi^\epsilon_{t+\tau}(x) - \phi^\epsilon_{t}(x) 
= &\int_{t}^{t+\tau}
v^\epsilon_s(\phi^\epsilon_s(x)) ds +
\sum_{k \in \N} \int_{t}^{t+\tau} 
\sigma_k(\phi^\epsilon_s(x))\eta^{\epsilon,k}_s ds
\\
&+\int_{t}^{t+\tau}
z_s(\phi^\epsilon_s(x)) ds 
+ \sqrt{2\nu} (w_{t+\tau}-w_t),
\end{align*}
therefore, by assumption (A2) we have 
\begin{align*}
\sup_{\substack{t+\tau \leq T}}
\| \phi^\epsilon_{t+\tau} - \phi^\epsilon_t \|_{L^{\infty}(\mathbb{T}^2,\mathbb{T}^2)}
\lesssim\,
&\tau  +
\tau\sum_{k \in \N} \|\sigma_k\|_{L^\infty(\T)}
\sup_{s\in[0,T]} |\eta^{\epsilon,k}_s|
\\
&+ \tau \sup_{s\in[0,T]}\|\zeta^\epsilon_s\|_{L^\infty(\T)}
+ K_\alpha \tau^\alpha,
\end{align*}
where $K_\alpha$ denotes the $\alpha$-H\"older constant of $w$.
The thesis follows easily by (A3), \eqref{eq:OU} and \autoref{lem:zeta_log}.
\end{proof}

\begin{lem} \label{lem:phi_eps_incr_bis}
Assume (A1)-(A3). Then for every fixed $p\geq 1$ we have, uniformly in $n=0,\dots,T/\delta -1$:
\begin{align*}
\hE{ \| \phi^\epsilon_{(n+1)\delta} - \phi^\epsilon_{n\delta} \|^p_{L^{\infty}(\mathbb{T}^2,\mathbb{T}^2)}}
&\lesssim
\delta^{2p} \epsilon^{-p} \log^p(1+\epsilon^{-1})
\\
&\quad+
\delta^{p(1+\alpha)} \epsilon^{-p/2} \log^{p/2}(1+\epsilon^{-1}) 
\\
&\quad+ 
\delta^{p/2} + \epsilon^{p/2} \log^{p/2}(1+\epsilon^{-1}).
\end{align*}
\end{lem}
\begin{proof}
The increment $\phi^\epsilon_{(n+1)\delta}(x) - \phi^\epsilon_{n\delta}(x)$ can be written as
\begin{align*}
\phi^\epsilon_{(n+1)\delta}(x) - \phi^\epsilon_{n\delta}(x)
= &\int_{n\delta}^{(n+1)\delta}
v^\epsilon_s(\phi^\epsilon_s(x)) ds \\
&+
\sum_{k \in \N}\int_{n\delta}^{(n+1)\delta}
 \left( \sigma_k(\phi^\epsilon_s(x))-\sigma_k(\phi^\epsilon_{n\delta}(x)) \right)\eta^{\epsilon,k}_s ds \\
&+
\sum_{k \in \N} \int_{n\delta}^{(n+1)\delta}
\sigma_k(\phi^\epsilon_{n\delta}(x)) \eta^{\epsilon,k}_s ds
\\
&\quad +\int_{n\delta}^{(n+1)\delta}
z^\epsilon_s(\phi^\epsilon_s(x)) ds 
+ \sqrt{2\nu} (w_{(n+1)\delta}-w_{n\delta}).
\end{align*}
The first, fourth and fifth term are easy. The second one is bounded in $L^\infty(\mathbb{T}^2,\mathbb{T}^2)$ uniformly in $n$ by
\begin{align*}
\int_{0}^{\delta}
\sum_{k \in \mathbb{N}} 
\|\nabla \sigma_k \|_{L^\infty(\mathbb{T}^2,\mathbb{R}^4)}
\sup_{\substack{t+s \leq T}}
\| \phi^\epsilon_{t+s} - \phi^\epsilon_t \|_{L^{\infty}(\mathbb{T}^2,\mathbb{T}^2)}
\sup_{s\in[0,T]} |\eta^{\epsilon,k}_s| ds,
\end{align*}
and by (A3) and H\"older inequality with exponent $q>1$
\begin{align*}
&\hE{ \left( \int_{0}^{\delta}
\sum_{k \in \mathbb{N}} 
\|\nabla \sigma_k \|_{L^\infty(\mathbb{T}^2,\mathbb{R}^4)}
\sup_{\substack{t+s \leq T}}
\| \phi^\epsilon_{t+s} - \phi^\epsilon_t \|_{L^{\infty}(\mathbb{T}^2,\mathbb{T}^2)}
\sup_{s\in[0,T]} |\eta^{\epsilon,k}_s| ds
\right)^p } \\
&\leq
\delta^{p-1}
\left( \sum_{k \in \mathbb{N}} 
\|\nabla \sigma_k \|_{L^\infty(\mathbb{T}^2,\mathbb{R}^4)} \right)^{p-1} 
\int_{0}^{\delta}
\sum_{k \in \mathbb{N}} 
\|\nabla \sigma_k \|_{L^\infty(\mathbb{T}^2,\mathbb{R}^4)}\\
&\quad \times
\hE{ \sup_{\substack{t+s \leq T}}
\| \phi^\epsilon_{t+s} - \phi^\epsilon_t \|^{pq}_{L^{\infty}(\mathbb{T}^2,\mathbb{T}^2)} }^{1/q} 
\hE{ \sup_{s\in[0,T]} |\eta^{\epsilon,k}_s|^{pq'} }^{1/q'} ds 
\\
&\lesssim
\delta^{p-1}
\int_{0}^{\delta}
\left( s^p \epsilon^{-p} \log^p(1+\epsilon^{-1}) ds
+
s^{p\alpha} \epsilon^{-p/2} \log^{p/2}(1+\epsilon^{-1}) \right) ds
\\
&\lesssim
\delta^{2p} \epsilon^{-p} \log^p(1+\epsilon^{-1})
+
\delta^{p(1+\alpha)} \epsilon^{-p/2} \log^{p/2}(1+\epsilon^{-1}).
\end{align*}
The third term is bounded in $L^\infty(\mathbb{T}^2,\mathbb{R}^2)$ by
\begin{align*}
\sum_{k \in \mathbb{N}}
\| \sigma_k \|_{L^\infty(\mathbb{T}^2,\mathbb{R}^2)}
\left| 
\int_{n\delta}^{(n+1)\delta}
\eta^{\epsilon,k}_s ds \right|
&=
\sum_{k \in \mathbb{N}}
\| \sigma_k \|_{L^\infty(\mathbb{T}^2,\mathbb{R}^2)}
\left| 
W^{k}_{(n+1)\delta}-W^{k}_{n\delta}\right|
\\
&\quad+ \sum_{k \in \mathbb{N}}
\| \sigma_k \|_{L^\infty(\mathbb{T}^2,\mathbb{R}^2)}
\epsilon
\left| 
\eta^{\epsilon,k}_{(n+1)\delta}-\eta^{\epsilon,k}_{n\delta}  \right|,
\end{align*}
from which we deduce as usual
\begin{align*}
\hE{\left( \sum_{k \in \mathbb{N}}
\| \sigma_k \|_{L^\infty(\mathbb{T}^2,\mathbb{R}^2)}
\left| 
\int_{n\delta}^{(n+1)\delta}
\eta^{\epsilon,k}_s ds \right|
\right)^p } 
\lesssim
\delta^{p/2} + \epsilon^{p/2} \log^{p/2}(1+\epsilon^{-1}).
\end{align*}
Putting all together, the thesis follows.
\end{proof}

\begin{proof}[Proof of \autoref{prop:old}]
For any given $t \in [0,T]$, let $\lfloor t \rfloor \eqqcolon m\delta$ be the largest multiple of $\delta$ strictly smaller than $t$.
We can therefore decompose 
\begin{align*}
\sum_{k \in \N}
\int_0^t
\sigma_k(\phi^\epsilon_s(x)) \eta^{\epsilon,k}_s ds
&=
\sum_{k \in \N}
\int_0^{m\delta}
\sigma_k(\phi^\epsilon_s(x)) \eta^{\epsilon,k}_s ds
+
\sum_{k \in \N}
\int_{m\delta}^t
\sigma_k(\phi^\epsilon_s(x)) \eta^{\epsilon,k}_s ds
\\
&=
\sum_{j = 1}^6
\sum_{n = 0}^{m-1} I^\epsilon_j(n)
+
\sum_{k \in \N}
\int_{m\delta}^t
\sigma_k(\phi^\epsilon_s(x)) \eta^{\epsilon,k}_s ds,
\end{align*}
and in a similar fashion
\begin{align*}
\sum_{k \in \N}
\int_0^t
\sigma_k(\phi^\epsilon_s(x)) \circ dW^k_s
&=
\sum_{k \in \N}
\int_0^{m\delta}
\sigma_k(\phi^\epsilon_s(x)) \circ dW^k_s
+
\sum_{k \in \N}
\int_{m\delta}^t
\sigma_k(\phi^\epsilon_s(x)) \circ dW^k_s
\\
&=
\sum_{j = 1}^4
\sum_{n = 0}^{m-1} J^\epsilon_j(n)
+
\sum_{k \in \N}
\int_{m\delta}^t
\sigma_k(\phi^\epsilon_s(x)) \circ dW^k_s.
\end{align*}

By \eqref{eq:OU}, the following estimate holds true
\begin{align*}
\hE{ \sup_{\substack{m = 0,\dots,T/\delta-1\\t \leq \delta}} \left\|
\sum_{k \in \N} \int_{m\delta}^t
\sigma_k(\phi^\epsilon_s(\cdot)) \eta^{\epsilon,k}_s ds
\right\|_{L^1(\T,\R^2)}
}
\lesssim
\delta \epsilon^{-1/2} \log^{1/2}(1+\epsilon^{-1}).
\end{align*}

Also, by (A3) and Kolmogorov continuity Theorem, for every fixed $\alpha \in (0,1/2)$ we have
\begin{align*}
\hE{ \sup_{\substack{m = 0,\dots,T/\delta-1\\t \leq \delta}} \left\|
\sum_{k \in \N}
\int_{m\delta}^t
\sigma_k(\phi^\epsilon_s(\cdot)) \circ dW^k_s
\right\|_{L^1(\T,\R^2)}
}
\lesssim
\delta^\alpha.
\end{align*}

Finally, by calculations similar to those performed in Lemma 4.6 and Lemma 4.7 of \cite{FlPa21}, for every fixed $\alpha \in (0,1/2)$
\begin{align*}
\hE{ \sup_{m = 0,\dots,T/\delta-1}
\left\|
\sum_{j = 1}^6
\sum_{n = 0}^{m-1} I^\epsilon_j(n)
-
\sum_{j = 1}^4
\sum_{n = 0}^{m-1} J^\epsilon_j(n)
\right\|_{L^1(\T,\R^2)}
}
\\
\lesssim 
\delta \epsilon^{-1/2}\log^{3/2}(1+\epsilon^{-1})
+
\delta^{\alpha-1} \epsilon^{1/2}\log(1+\epsilon^{-1})
\\
\delta^2 \epsilon^{-3/2}\log^{3/2}(1+\epsilon^{-1})
+
\delta^{1+\alpha} \epsilon^{-1}\log(1+\epsilon^{-1})
+
\delta^{\alpha}
.
\end{align*}
We conclude the proof fixing $\alpha$ close to $1/2$ so that $(1+\alpha)^{-1} < 3/4 < (2-2\alpha)^{-1}$, for instance $\alpha = 3/8$, and optimizing over $\delta$: for $\delta = \epsilon^{16/21}\log^{-4/21}(1+\epsilon^{-1})$, it follows the desired inequality
\begin{align*}
\hE{ \sup_{t \in [0,T]}\left\|
\sum_{k \in \N}
\int_0^t
\sigma_k(\phi^\epsilon_s(\cdot)) \eta^{\epsilon,k}_s ds
-
\sum_{k \in \N}
\int_0^t
\sigma_k(\phi^\epsilon_s(\cdot)) \circ dW^k_s
\right\|_{L^1(\T,\R^2)}}
\\
\lesssim
\epsilon^{1/42} \log^{47/42}(1+\epsilon^{-1}).
\end{align*}
\end{proof}

\subsection{Proof of \autoref{prop:sup_z}}
Recall the content of \autoref{prop:sup_z}: we need to prove, under assumptions (A1)-(A5)
\begin{align*}
\hE{\sup_{t \in [0,T]}
\left\|\int_0^t
z^\epsilon_s(\phi^\epsilon_s(\cdot)) ds\right\|_{L^\infty(\T,\R^2)}}
\lesssim
\epsilon^{1/12} \log^{11/12}(1+\epsilon^{-1}).
\end{align*}

Comparing the desired inequality with \autoref{lem:zeta_log}, one realizes that time integration of the process $z^\epsilon_s(\phi^\epsilon_s(x))$ allows a better control due to cancellation of opposite-sign oscillations, even if the latter may become of large magnitude for $\epsilon$ going to zero.

Concerning the strategy of the proof, in the first place we prove the following:
\begin{lem} \label{lem:z}
For every fixed $t \in [0,T]$ it holds
\begin{align*}
\hE{\left\|\int_0^t
z^\epsilon_s(\phi^\epsilon_s(\cdot)) ds\right\|_{L^\infty(\T,\R^2)}}
\lesssim
\epsilon^{1/6} \log^{5/6}(1+\epsilon^{-1}).
\end{align*}
\end{lem}

Having at hands the previous result, the proof of \autoref{prop:sup_z} goes as follows: for some parameter $\delta=T/m>0$, $m \in \N$ to be chosen, write
\begin{align*}
\sup_{t \in [0,T]}\left\|\int_0^t
z^\epsilon_s(\phi^\epsilon_s(\cdot)) ds\right\|_{L^\infty(\T,\R^2)}
&\leq
\sup_{n=0,\dots,m-1} 
\left\|\int_0^{n\delta}
z^\epsilon_s(\phi^\epsilon_s(\cdot)) ds\right\|_{L^\infty(\T,\R^2)}
\\
&\quad+
\sup_{\substack{n=0,\dots,m-1\\t \leq \delta}} 
\left\|\int_{n\delta}^{n\delta+t}
z^\epsilon_s(\phi^\epsilon_s(\cdot)) ds\right\|_{L^\infty(\T,\R^2)}
\\
&\leq
\sum_{n=0}^{m-1} 
\left\|\int_0^{n\delta}
z^\epsilon_s(\phi^\epsilon_s(\cdot)) ds\right\|_{L^\infty(\T,\R^2)}
\\
&\quad+
\delta
\sup_{s \in [0,T]} 
\|z^\epsilon_s(\phi^\epsilon_s(\cdot))\|_{L^\infty(\T,\R^2)}.
\end{align*}
Hence, by \autoref{lem:zeta_log} and \autoref{lem:z}
\begin{align*}
\hE{\sup_{t \in [0,T]}
\left\|\int_0^t
z^\epsilon_s(\phi^\epsilon_s(\cdot)) ds\right\|_{L^\infty(\T,\R^2)}}
&\leq
\sum_{n=0}^{m-1}
\hE{ 
\left\|\int_0^{n\delta}
z^\epsilon_s(\phi^\epsilon_s(\cdot)) ds\right\|_{L^\infty(\T,\R^2)}}
\\
&\quad+
\delta
\hE{\sup_{s \in [0,T]} 
\|z^\epsilon_s(\phi^\epsilon_s(\cdot))\|_{L^\infty(\T,\R^2)}}
\\
&\lesssim
\delta^{-1} \epsilon^{1/6} \log^{5/6}(1+\epsilon^{-1})
+ \delta \log(1+\epsilon^{-1}),
\end{align*}
and the thesis follows by optimizing the choice of $\delta$.

\begin{proof}[Proof of \autoref{lem:z}]
We will work with fixed $x \in \T$. The reader can easily check that all the inequalities present in the proof hold uniformly in $x$.
Recall $z^\epsilon_t = K \ast \zeta^\epsilon_t$, and for $\psi^\epsilon_{t,s}(x) \coloneqq \psi^\epsilon_s((\psi^\epsilon_t)^{-1}(x))$ the formula
\begin{align*}
\zeta^\epsilon_t(x)
&=
-\int_0^t
e^{-\epsilon^{-1}(t-s)}
((v^\epsilon_s + K \ast \zeta^\epsilon_s)\cdot\nabla {\theta}^\epsilon_s )(\psi^\epsilon_{t,s}(x)) ds
\\
&\quad-
\int_0^t
e^{-\epsilon^{-1}(t-s)}
((K \ast \theta^\epsilon_s)\cdot\nabla {\theta}^\epsilon_s )(\psi^\epsilon_{t,s}(x)) ds.
\end{align*}

For notational simplicity let $\Theta^\epsilon_s \coloneqq (K \ast \theta^\epsilon_s)\cdot\nabla {\theta}^\epsilon_s$, and rewrite
\begin{align*}
\zeta^\epsilon_t(x)
&=
-\int_0^t
e^{-\epsilon^{-1}(t-s)}
((v^\epsilon_s + K \ast \zeta^\epsilon_s)\cdot\nabla {\theta}^\epsilon_s )(\psi^\epsilon_{t,s}(x)) ds
\\
&\quad-
\int_0^t
e^{-\epsilon^{-1}(t-s)}
\left(
\Theta^\epsilon_s(\psi^\epsilon_{t,s}(x)) -
\Theta^\epsilon_s(x) \right) ds
\\
&\quad-
\int_0^t
e^{-\epsilon^{-1}(t-s)}
\Theta^\epsilon_s(x) ds
\\
&\quad \eqqcolon
\zeta^{\epsilon,1}_t(x) + \zeta^{\epsilon,2}_t(x) + \zeta^{\epsilon,3}_t(x).
\end{align*}

Let us focus on the terms $\zeta^{\epsilon,j}$, $j=1,2,3$ separately. Concerning $\zeta^{\epsilon,1}$,
\begin{align*}
\|\zeta^{\epsilon,1}_t \|_{L^\infty(\T)}
&\lesssim
\int_0^t e^{-\epsilon^{-1}(t-s)} ds \left(1+\sup_{s \in [0,T]} \|\zeta^\epsilon_s\|_{L^\infty(\T)}\right)
\\
&\qquad \times \sup_{s \in [0,T]}\| \nabla \theta^\epsilon_s \|_{L^\infty(\T,\R^2)},
\end{align*}
and thus the following holds by assumption (A2) and \autoref{lem:zeta_log}
\begin{align} \label{eq:zeta1}
\sup_{t \in [0,T]} \hE{\|\zeta^{\epsilon,1}_t\|_{L^\infty(\T)}}
\lesssim
\epsilon^{1/2} \log^{3/2}(1+\epsilon^{-1}).
\end{align}

Moving to $\zeta^{\epsilon,2}$, notice that $|\psi^\epsilon_{t,s}(x)-x| = |\psi^\epsilon_{t,s}(x)-\psi^\epsilon_{t,t}(x)|$, and letting $y=(\psi^\epsilon_t)^{-1}(x)$ we have
\begin{align*}
|\psi^\epsilon_{t,s}(x)-\psi^\epsilon_{t,t}(x)|
&=
|\psi^\epsilon_s(y)-\psi^\epsilon_t(y)|
\\
&\leq
\int_s^t
|v^\epsilon_r(\psi^\epsilon_r(y))| dr
+
\int_s^t
|u^\epsilon_r(\psi^\epsilon_r(y))| dr
\\
&\lesssim
|t-s| \left( 1+  
\sup_{r \in [0,T]} \|\zeta^\epsilon_r\|_{L^\infty(\T)}
+\sup_{r \in [0,T]} \|\theta^\epsilon_r\|_{L^\infty(\T)} \right),
\end{align*}
therefore
\begin{align*}
\|\zeta^{\epsilon,2}_t \|_{L^\infty(\T)}
&\lesssim
\int_0^t e^{-\epsilon^{-1}(t-s)} |t-s|  ds \sup_{s \in [0,T]}\| \nabla \Theta^\epsilon_s \|_{L^\infty(\T,\R^2)}
\\
&\qquad \times 
\left( 1+  
\sup_{r \in [0,T]} \|\zeta^\epsilon_r\|_{L^\infty(\T)}
+\sup_{r \in [0,T]} \|\theta^\epsilon_r\|_{L^\infty(\T)} \right),
\end{align*}
that implies
\begin{align} \label{eq:zeta2}
\sup_{t \in [0,T]}\hE{\|\zeta^{\epsilon,2}_t \|_{L^\infty(\T)}}
&\lesssim
\epsilon^{1/2} \log^{3/2}(1+\epsilon^{-1}).
\end{align}

Finally, let us consider the term $\zeta^{\epsilon,3}$, which requires a preliminary manipulation. Since $\theta^\epsilon_s(x) = \sum_{k \in \N} \sigma_k(x) \eta^{\epsilon,k}_s$, we can rewrite for every $x \in \T$
\begin{align*}
\Theta^\epsilon_s(x)
=
\sum_{k,h \in \N}
(\sigma_k\cdot\nabla \varsigma_h) (x)
\eta^{\epsilon,k}_s \eta^{\epsilon,h}_s
\eqqcolon
\sum_{k,h \in \N}
\Theta_{k,h}(x)
\eta^{\epsilon,k}_s \eta^{\epsilon,h}_s,
\end{align*}
where we have used $\sigma_k = K \ast \varsigma_k$ and $\Theta_{k,h} \coloneqq \sigma_k\cdot\nabla \varsigma_h$.
Also, rewrite:
\begin{align*}
\zeta^{\epsilon,3}_t(x)
&=
-
\int_0^t
e^{-\epsilon^{-1}(t-s)}
\Theta^\epsilon_s(x) ds
\\
&=
-
\sum_{k,h \in \N}
\Theta_{k,h}(x)
\int_0^t
e^{-\epsilon^{-1}(t-s)}
\eta^{\epsilon,k}_s \eta^{\epsilon,h}_s ds.
\end{align*}

By It\=o Formula, for every fixed $t$ and $k,h \in \N$ it holds
\begin{align*}
\eta^{\epsilon,k}_t \eta^{\epsilon,h}_t
&=
e^{-\epsilon^{-1} t }\eta^{\epsilon,k}_0 \eta^{\epsilon,h}_0
-
\epsilon^{-1} \int_0^t e^{-\epsilon^{-1}(t-s)}
\eta^{\epsilon,k}_s \eta^{\epsilon,h}_s ds
\\
&\quad
+\epsilon^{-1} \int_0^t e^{-\epsilon^{-1}(t-s)} \eta^{\epsilon,k}_s dW^h_s
+\epsilon^{-1} \int_0^t e^{-\epsilon^{-1}(t-s)} \eta^{\epsilon,h}_s dW^k_s
\\
&\quad
+\frac{\epsilon^{-2}}{2} \delta_{k,h} \int_0^t e^{-\epsilon^{-1}(t-s)} ds,
\end{align*}
with $\delta_{k,h}$ being the Kronecker delta function: $\delta_{k,h}=1$ if $k=h$ and $\delta_{k,h}=0$ if $k\neq h$.
Otherwise said:
\begin{align} \label{eq:int_eOU}
\int_0^t e^{-\epsilon^{-1}(t-s)}
\eta^{\epsilon,k}_s \eta^{\epsilon,h}_s ds
&=
\epsilon \left(  {e^{-\epsilon^{-1} t }\eta^{\epsilon,k}_0 \eta^{\epsilon,h}_0-\eta^{\epsilon,k}_t \eta^{\epsilon,h}_t} \right)
\\
&\quad \nonumber
\int_0^t e^{-\epsilon^{-1}(t-s)} \eta^{\epsilon,k}_s dW^h_s
+\int_0^t e^{-\epsilon^{-1}(t-s)} \eta^{\epsilon,h}_s dW^k_s
\\
&\quad
+\frac{1-e^{\epsilon^{-1} t}}{2} \delta_{k,h}. \nonumber
\end{align}

By \eqref{eq:int_eOU} and assumption (A5), for every $x \in \T$ we have
\begin{align*}
\zeta^{\epsilon,3}_t(x)
&=
\sum_{k,h \in \N}
\Theta^\epsilon_{k,h}(x)
\epsilon \left( \eta^{\epsilon,k}_t \eta^{\epsilon,h}_t -e^{-\epsilon^{-1} t }\eta^{\epsilon,k}_0 \eta^{\epsilon,h}_0 \right)
\\
&\quad-
\sum_{k,h \in \N}
\Theta^\epsilon_{k,h}(x)
\left( \int_0^t e^{-\epsilon^{-1}(t-s)} \eta^{\epsilon,k}_s dW^h_s
+\int_0^t e^{-\epsilon^{-1}(t-s)} \eta^{\epsilon,h}_s dW^k_s\right),
\end{align*} 
and therefore we can rewrite
\begin{align*}
\int_0^t
(K \ast \zeta^{\epsilon,3}_s)(\phi^\epsilon_s(x)) ds
&=
\sum_{k,h \in \N}
\int_0^t
(K \ast \Theta^\epsilon_{k,h})(\phi^\epsilon_s(x))
\epsilon  \eta^{\epsilon,k}_s \eta^{\epsilon,h}_s 
ds
\\
&\quad-
\sum_{k,h \in \N}
\int_0^t
(K \ast \Theta^\epsilon_{k,h})(\phi^\epsilon_s(x))
\epsilon  e^{-\epsilon^{-1} s }\eta^{\epsilon,k}_0 \eta^{\epsilon,h}_0 
ds
\\
&\hspace{-112pt}\quad
- \sum_{k,h \in \N}
\int_0^t
(K \ast \Theta_{k,h})(\phi^\epsilon_s(x))
\left(
\int_0^s e^{-\epsilon^{-1}(s-r)} \eta^{\epsilon,k}_r dW^h_r
+\int_0^s e^{-\epsilon^{-1}(s-r)} \eta^{\epsilon,h}_r dW^k_r
\right) ds.
\end{align*} 

It holds
\begin{align*}
\hE{\left| \int_0^t
(K \ast \Theta^\epsilon_{k,h})(\phi^\epsilon_s(x))
\epsilon  e^{-\epsilon^{-1} s }\eta^{\epsilon,k}_0 \eta^{\epsilon,h}_0 
ds\right|}
&\lesssim
\epsilon \log(1+\epsilon^{-1}),
\end{align*}
\begin{align*}
&\hE{\left| \int_0^t
(K \ast \Theta^\epsilon_{k,h})(\phi^\epsilon_s(x))
\int_0^s e^{-\epsilon^{-1}(s-r)} \eta^{\epsilon,k}_r dW^h_r
ds \right|}
\\
&\quad=\hE{\left| \int_0^t \left(\int_r^t
(K \ast \Theta^\epsilon_{k,h})(\phi^\epsilon_s(x))
e^{-\epsilon^{-1}(s-r)} ds \right) \eta^{\epsilon,k}_r dW^h_r
\right|}
\\
&\quad\lesssim
\hE{\left| \int_0^t \left(\int_r^t
(K \ast \Theta^\epsilon_{k,h})(\phi^\epsilon_s(x))
e^{-\epsilon^{-1}(s-r)} ds \right) \eta^{\epsilon,k}_r dW^h_r
\right|^2}^{1/2}
\\
&\quad\lesssim
\hE{ \int_0^t \left(\int_r^t
(K \ast \Theta^\epsilon_{k,h})(\phi^\epsilon_s(x))
e^{-\epsilon^{-1}(s-r)} ds \right)^2  |\eta^{\epsilon,k}_r|^2 dr
}^{1/2}
\\
&\quad\lesssim
\epsilon^{1/2} \log^{1/2}(1+\epsilon^{-1}).
\end{align*}

The last non-trivial term is manipulated as follows. Let $\delta=t/m>0$, $m \in \N$ to be suitably chosen. We have
\begin{align} \label{eq:zeta3_last}
\sum_{k,h \in \N}
&\int_0^t
(K \ast \Theta_{k,h})(\phi^\epsilon_s(x))
\epsilon \eta^{\epsilon,k}_s \eta^{\epsilon,h}_s
ds
\\
&=
\sum_{k,h \in \N} 
\sum_{n=0}^{m-1}
\int_{n\delta}^{(n+1)\delta}
\left(
(K \ast \Theta_{k,h})(\phi^\epsilon_s(x))-
(K \ast \Theta_{k,h})(\phi^\epsilon_{n\delta}(x)) \right)
\epsilon \eta^{\epsilon,k}_s \eta^{\epsilon,h}_s
ds \nonumber
\\
&\quad+
\sum_{k,h \in \N} 
\sum_{n=0}^{m-1}
(K \ast \Theta_{k,h})(\phi^\epsilon_{n\delta}(x))
\int_{n\delta}^{(n+1)\delta}
\epsilon \eta^{\epsilon,k}_s \eta^{\epsilon,h}_s
ds. \nonumber
\end{align}
Recalling \eqref{eq:char_eps_intro}, for every $\alpha \in (0,1/2)$ it holds
\begin{align*}
|\phi^\epsilon_t(x)-\phi^\epsilon_s(x)|
&\leq
\int_s^t
|v^\epsilon_r(\phi^\epsilon_r(x))| dr
+
\int_s^t
|u^\epsilon_r(\phi^\epsilon_r(x))| dr
+
\sqrt{2\nu} (w_t-w_s)
\\
&\lesssim
|t-s| \left( 1+  
\sup_{r \in [0,T]} \|\zeta^\epsilon_r\|_{L^\infty(\T)}
+\sup_{r \in [0,T]} \|\theta^\epsilon_r\|_{L^\infty(\T)} \right)
+
|t-s|^\alpha,
\end{align*}
which implies
\begin{align} \label{eq:delta1}
&\hE{\left| \sum_{k,h \in \N} 
\sum_{n=0}^{m-1}
\int_{n\delta}^{(n+1)\delta}
\left(
(K \ast \Theta_{k,h})(\phi^\epsilon_s(x))-
(K \ast \Theta_{k,h})(\phi^\epsilon_{n\delta}(x)) \right)
\epsilon \eta^{\epsilon,k}_s \eta^{\epsilon,h}_s
ds \right|}
\\
&\quad\lesssim
\delta \epsilon^{-1/2} \log^{3/2}(1+\epsilon^{-1})
+
\delta^\alpha \log(1+\epsilon^{-1}) \nonumber.
\end{align}

Also, we can apply It\=o Formula again to find an alternative representation for the time integral of the quadratics $\eta^{\epsilon,k}_s \eta^{\epsilon,h}_s$, similar to \eqref{eq:int_eOU}. Indeed,
\begin{align*} 
\eta^{\epsilon,k}_{(n+1)\delta} \eta^{\epsilon,h}_{(n+1)\delta}
- \eta^{\epsilon,k}_{n\delta} \eta^{\epsilon,h}_{n\delta}
&=
-2\epsilon^{-1} \int_{n\delta}^{(n+1)\delta} 
\eta^{\epsilon,k}_t \eta^{\epsilon,h}_t dt
\\
&\quad
+\epsilon^{-1} \int_{n\delta}^{(n+1)\delta} \eta^{\epsilon,k}_t dW^h_t
+\epsilon^{-1} \int_{n\delta}^{(n+1)\delta} \eta^{\epsilon,h}_t dW^k_t 
\\
&\quad
+\frac{\epsilon^{-2} \delta}{2} \delta_{k,h}, 
\end{align*}
and rearranging the terms we obtain
\begin{align} \label{eq:int_OU}
\int_{n\delta}^{(n+1)\delta} 
\epsilon \eta^{\epsilon,k}_t \eta^{\epsilon,h}_t dt
&=
\frac{\epsilon^2}{2} \left( \eta^{\epsilon,k}_{n\delta} \eta^{\epsilon,h}_{n\delta} - \eta^{\epsilon,k}_{(n+1)\delta} \eta^{\epsilon,h}_{(n+1)\delta} \right)
\\
&\quad
+\frac{\epsilon}2 \int_{n\delta}^{(n+1)\delta} \eta^{\epsilon,k}_t dW^h_t
+\frac{\epsilon}2 \int_{n\delta}^{(n+1)\delta} \eta^{\epsilon,h}_t dW^k_t
+\frac{\delta}{4} \delta_{k,h} \nonumber.
\end{align}

Finally, making use of \eqref{eq:int_OU} above and assumption (A5) we can rewrite
\begin{align*}
\sum_{k,h \in \N} 
&\sum_{n=0}^{m-1}
(K \ast \Theta_{k,h})(\phi^\epsilon_{n\delta}(x))
\int_{n\delta}^{(n+1)\delta}
\epsilon \eta^{\epsilon,k}_s \eta^{\epsilon,h}_s
ds
\\
&=
\sum_{k,h \in \N} 
\sum_{n=0}^{m-1}
(K \ast \Theta_{k,h})(\phi^\epsilon_{n\delta}(x))
\frac{\epsilon^2}{2} \left( \eta^{\epsilon,k}_{n\delta} \eta^{\epsilon,h}_{n\delta} - \eta^{\epsilon,k}_{(n+1)\delta} \eta^{\epsilon,h}_{(n+1)\delta} \right)
\\
&\hspace{-10pt}\quad+
\sum_{k,h \in \N} 
\sum_{n=0}^{m-1}
(K \ast \Theta_{k,h})(\phi^\epsilon_{n\delta}(x))
\left(\frac{\epsilon}2 \int_{n\delta}^{(n+1)\delta} \eta^{\epsilon,k}_t dW^h_t
+\frac{\epsilon}2 \int_{n\delta}^{(n+1)\delta} \eta^{\epsilon,h}_t dW^k_t \right).
\end{align*}

We have
\begin{align} \label{eq:delta2}
\hE{\left| \sum_{k,h \in \N} 
\sum_{n=0}^{m-1}
(K \ast \Theta_{k,h})(\phi^\epsilon_{n\delta}(x))
\frac{\epsilon^2}{2} \left( \eta^{\epsilon,k}_{n\delta} \eta^{\epsilon,h}_{n\delta} - \eta^{\epsilon,k}_{(n+1)\delta} \eta^{\epsilon,h}_{(n+1)\delta} \right) \right|}
\\
\lesssim \nonumber
\delta^{-1} \epsilon \log(1+\epsilon^{-1}),
\end{align}
and 
\begin{align} \label{eq:delta3}
&\hE{\left| 
\sum_{k,h \in \N} 
\sum_{n=0}^{m-1}
(K \ast \Theta_{k,h})(\phi^\epsilon_{n\delta}(x))
\frac{\epsilon}2 \int_{n\delta}^{(n+1)\delta} \eta^{\epsilon,k}_t dW^h_t
\right|}
\\
&\quad\lesssim 
\sum_{n=0}^{m-1}
\epsilon
\hE{\left| 
\int_{n\delta}^{(n+1)\delta} \eta^{\epsilon,k}_t dW^h_t
\right|^2}^{1/2}
\lesssim
\delta^{-1/2} \epsilon^{1/2} \log^{1/2}(1+\epsilon^{-1}). \nonumber
\end{align}

It only remains to choose $\delta$ in a suitable way, so that all the terms \eqref{eq:delta1}, \eqref{eq:delta2} and  \eqref{eq:delta2} are infinitesimal in the limit $\epsilon \to 0$. 
Taking for instance $\alpha=1/3$ and optimizing over $\delta$ gives
\begin{align} \label{eq:zeta3}
\hE{\left| \int_0^t (K \ast \zeta^{\epsilon,3}_s)(\phi^\epsilon_s(x)) ds \right|}
\lesssim
\epsilon^{1/6} \log^{5/6}(1+\epsilon^{-1}).
\end{align}
Considering \eqref{eq:zeta1}, \eqref{eq:zeta2} and \eqref{eq:zeta3}, we finally get the desired estimate: the proof is complete.
\end{proof}

\section{Convergence of characteristics} \label{sec:conv_char}
In this section we prove our first major result \autoref{thm:char}.

We take the opportunity to point out a mistake in \cite[Lemma 3.8]{FlPa21}, where BDG inequality was applied incorrectly. The present proof also corrects this previous mistake, and it is based on It\=o Formula for a smooth approximation $g_\delta(x)$ of the absolute value $|x|$. 

\begin{proof}[Proof of \autoref{thm:char}]
The strategy of the proof is very similar to that of \cite[Proposition 4.1]{FlPa21}. Indeed, the difference $\phi^\epsilon-\phi$ solves $\hat{\PP}$-a.s. for every $t \in [0,T]$ and $x \in \T$:
\begin{align*}
\phi_t^\epsilon(x)-\phi_t(x)
&=
\int_0^t
v^\epsilon_s(\phi^\epsilon_s(x)) ds
-
\int_0^t
v_s(\phi^\epsilon_s(x)) ds
\\
&\quad
+
\int_0^t
v_s(\phi^\epsilon_s(x)) ds
-
\int_0^t
v_s(\phi_s(x)) ds
\\
&\quad
+
\sum_{k \in \N}
\int_0^t
\sigma_k(\phi^\epsilon_s(x)) \eta^{\epsilon,k}_s ds
-
\sum_{k \in \N}
\int_0^t
\sigma_k(\phi^\epsilon_s(x))\circ dW^k_s
\\
&\quad
+
\sum_{k \in \N}
\int_0^t
\sigma_k(\phi^\epsilon_s(x)) \circ dW^k_s
-
\sum_{k \in \N}
\int_0^t
\sigma_k(\phi_s(x))\circ dW^k_s
\\
&\quad
+\int_0^t
z^\epsilon_s(\phi^\epsilon_s(x)) ds.
\end{align*} 

For $\delta>0$, introduce the smooth function $g_\delta:\R^2 \to \R$ defined by $g_\delta(x) \coloneqq (|x|^2+\delta)^{1/2}$.
It holds $\partial_{x_j} g_\delta(x)= x_j g_\delta(x)^{-1}$ and $\partial_{x_j}\partial_{x_i} g_\delta(x)= g_\delta(x)^{-1}(\delta_{i,j}-x_i x_j g_\delta(x)^{-2})$ for every $x \in \R^2$ and $j=1,2$, and moreover $|x| \leq g_\delta(x) \leq |x| + \delta^{1/2}$.

Denote
\begin{align*}
R^\epsilon_t(x) 
&\coloneqq 
\int_0^t z^\epsilon_s(\phi^\epsilon_s(x))ds
+
\sum_{k \in \N}
\int_0^t
\sigma_k(\phi^\epsilon_s(x)) \eta^{\epsilon,k}_s ds
-
\sum_{k \in \N}
\int_0^t
\sigma_k(\phi^\epsilon_s(x))\circ dW^k_s,
\end{align*}
and 
\begin{align*}
Z^\epsilon_t(x) \coloneqq \phi_t^\epsilon(x)-\phi_t(x)-R^\epsilon_t(x),
\end{align*}
both seen as functions on the whole plane $\R^2$.
Applying It\=o Formula to $g_\delta(Z^\epsilon_t(x))$ yields:
{\small
\begin{align*}
d g_\delta(Z^\epsilon_t(x)) 
&= 
Z^\epsilon_t(x) g_\delta(Z^\epsilon_t(x))^{-1} \cdot
\left( 
v^\epsilon_t(\phi^\epsilon_t(x))-v_t(\phi^\epsilon_t(x))
\right) dt
\\
&\quad+
Z^\epsilon_t(x) g_\delta(Z^\epsilon_t(x))^{-1} \cdot
\left(
v_t(\phi^\epsilon_t(x))-v_t(\phi_t(x))
\right) dt
\\
&\quad+
\sum_{k \in \N}
Z^\epsilon_t(x) g_\delta(Z^\epsilon_t(x))^{-1} \cdot
\left(
\sigma_k(\phi^\epsilon_t(x))-\sigma_k(\phi_t(x)) 
\right)
dW^k_t
\\
&\quad+
Z^\epsilon_t(x) g_\delta(Z^\epsilon_t(x))^{-1} \cdot
\left(
c(\phi^\epsilon_t(x))-c(\phi_t(x)) 
\right)
dt
\\
&\quad
+\sum_{k \in \N}
\sum_{i,j = 1}^2
g_\delta(Z^\epsilon_t(x))^{-1}
(\delta_{i,j}-(Z^\epsilon_t(x))^i (Z^\epsilon_t(x))^j g_\delta(Z^\epsilon_t(x))^{-2})
\\
&\qquad \times
\left(
\sigma_k(\phi^\epsilon_t(x))-\sigma_k(\phi_t(x)) 
\right)^i
\left(
\sigma_k(\phi^\epsilon_t(x))-\sigma_k(\phi_t(x)) 
\right)^j
dt
,
\end{align*}}
and therefore 
\begin{align*}
\hE{\left|\phi^\epsilon_t(x)-\phi^\epsilon_t(x)\right|}
&\leq
\hE{|Z^\epsilon_t(x)|}
+
\hE{\left|R^\epsilon_t(x)\right|}
\leq
\hE{g_\delta(Z_t)}
+
\hE{\left|R^\epsilon_t(x)\right|}
\\
&\lesssim
\delta^{1/2}
+
\hE{\left|R^\epsilon_t(x)\right|}
+
\hE{ \int_0^t
\left|
v^\epsilon_s(\phi^\epsilon_s(x))-v_s(\phi^\epsilon_s(x))
\right| ds} 
\\
&\quad+
\hE{ \int_0^t
\left|
v_s(\phi^\epsilon_s(x))-v_s(\phi_s(x))
\right| ds}  
\\
&\quad+
\hE{ \int_0^t
\left|
\phi^\epsilon_s(x)-\phi_s(x) 
\right| ds }
+ 
\delta^{-1/2}
\hE{\sup_{t \in [0,T]} \left|R^\epsilon_t(x)\right|},
\end{align*}
where in the last line we have used $g_\delta(Z^\epsilon_s(x))^{-1} \leq \delta^{-1/2}$ and $\left|
\phi^\epsilon_s(x)-\phi_s(x) 
\right| \lesssim |Z^\epsilon_s(x)| + \left|R^\epsilon_s(x)\right|$.

Taking the integral over $x \in \T$ and using assumptions (A2), (A4), concavity of the function $\gamma$, Jensen inequality, \autoref{prop:old} and \autoref{prop:sup_z} we get
\begin{align*}
\hE{ \| \phi^\epsilon_t - \phi_t \|_{L^1(\T,\T)}}
&\lesssim
\delta^{1/2}
+
\delta^{-1/2} \epsilon^{1/42} \log^{47/42}(1+\epsilon^{-1})
+
c_\epsilon
\\
&\quad+
\int_0^t
\gamma \left( \hE{ \| \phi^\epsilon_s - \phi_s \|_{L^1(\T,\T)}} \right) ds
\end{align*}
uniformly in $t \in [0,T]$ and $\delta>0$.
Taking $\delta = \epsilon^{1/42} \log^{47/42}(1+\epsilon^{-1})$ we deduce 
the desired result by \autoref{lem:comp}.
\end{proof}

\section{Convergence of large-scale dynamics} \label{sec:conv_large}

Recall the representation formulas for the solutions of \eqref{eq:large_eps_intro} and \eqref{eq:large_intro}
\begin{align*}
\Xi^\epsilon_t
&=
\tE{ \Xi_0 \circ (\phi^\epsilon_t)^{-1} + \int_0^t q^\epsilon_s \circ \phi^\epsilon_s \circ(\phi^\epsilon_t)^{-1}ds},
\\
\Xi_t
&=
\tE{ \Xi_0 \circ (\phi_t)^{-1} + \int_0^t q_s \circ \phi_s \circ(\phi_t)^{-1}ds},
\end{align*}
with $\phi^\epsilon$ and $\phi$ solving respectively \eqref{eq:char_eps_intro} and \eqref{eq:char_intro}.

As made clear by the following proof, these representation formulas are the key ingredient needed to show \autoref{thm:conv_large}, thus justifying our \autoref{def:sol} in terms of these identities.

\begin{proof}[Proof of \autoref{thm:conv_large}]
Let $f \in L^1(\T)$ and $t \in [0,T]$.
We have
\begin{align*}
&\left|\int_\T \Xi^\epsilon_t(x)f(x)dx - \int_\T \Xi_t(x)f(x)dx\right|
\\
&\leq
\left|\int_\T \tE{ \Xi_0((\phi^\epsilon_t)^{-1}(x))} f(x)dx - \int_\T \tE{ \Xi_0((\phi_t)^{-1}(x))}f(x)dx\right|
\\
&\quad+
\left|\int_\T \tE{ \int_0^t q^\epsilon_s(\phi^\epsilon_s((\phi^\epsilon_t)^{-1}(x)))ds} f(x)dx - \int_\T \tE{ \int_0^t q_s(\phi_s((\phi_t)^{-1}(x)))ds}f(x)dx\right|
\\
&=
\left|\tE{ \int_\T \Xi_0((\phi^\epsilon_t)^{-1}(x)) f(x)dx - \int_\T \Xi_0((\phi_t)^{-1}(x))f(x)dx}\right|
\\
&\quad+
\left|\tE{ \int_\T  \int_0^t q^\epsilon_s(\phi^\epsilon_s((\phi^\epsilon_t)^{-1}(x)))ds f(x)dx - \int_\T \int_0^t q_s(\phi_s((\phi_t)^{-1}(x)))ds f(x)dx}\right|
\\
&=
\left|\tE{ \int_\T \Xi_0(y) f(\phi^\epsilon_t(y))dy - \int_\T \Xi_0(y)f(\phi_t(y))dy}\right|
\\
&\quad+
\left|\tE{\int_0^t \int_\T  q^\epsilon_s(\phi^\epsilon_s(y)) f(\phi_t^\epsilon(y))dy ds - \int_0^t \int_\T  q_s(\phi_s(y))f(\phi_t(y)dyds }\right|.
\end{align*}

Taking expectation with respect to $\PP$, the first summand is bounded by
\begin{align} \label{eq:bound1}
\E{\left|\tE{ \int_\T \Xi_0(y) f(\phi^\epsilon_t(y))dy - \int_\T \Xi_0(y)f(\phi_t(y))dy}\right|} \nonumber
\\
\leq
\|\Xi_0\|_{L^\infty(\T)} \hE{ \int_\T \left| f(\phi^\epsilon_t(y)) - f(\phi_t(y)) \right| dy}.
\end{align}
As for the second term, we can rewrite
\begin{align*}
\int_0^t \int_\T&  q^\epsilon_s(\phi^\epsilon_s(y)) f(\phi_t^\epsilon(y))dy ds - \int_0^t \int_\T  q_s(\phi_s(y))f(\phi_t(y))dyds
\\
&=
\int_0^t \int_\T  q^\epsilon_s(\phi^\epsilon_s(y)) f(\phi_t^\epsilon(y))dy ds - \int_0^t \int_\T  q^\epsilon_s(\phi^\epsilon_s(y)) f(\phi_t(y))dyds
\\
&\quad+
\int_0^t \int_\T  q^\epsilon_s(\phi^\epsilon_s(y))  f(\phi_t(y)) dy ds - \int_0^t \int_\T  q_s(\phi^\epsilon_s(y))f(\phi_t(y))dyds
\\
&\quad+
\int_0^t \int_\T  q_s(\phi^\epsilon_s(y))  f(\phi_t(y)) dy ds - \int_0^t \int_\T  q_s(\phi_s(y))f(\phi_t(y))dyds,
\end{align*}
with estimates
\begin{align} \label{eq:bound2}
&\hE{\left|\int_0^t \int_\T  q^\epsilon_s(\phi^\epsilon_s(y)) f(\phi_t^\epsilon(y))dy ds - \int_0^t \int_\T  q^\epsilon_s(\phi^\epsilon_s(y)) f(\phi_t(y))dyds \right|} \nonumber
\\
&\qquad
\leq
\int_0^t \|q^\epsilon_s\|_{L^\infty(\T)} ds 
\hE{\int_\T |f(\phi_t^\epsilon(y)) - f(\phi_t(y)) |dy} ;
\end{align}
\begin{align} \label{eq:bound3}
&\hE{\left| \int_0^t \int_\T  q^\epsilon_s(\phi^\epsilon_s(y))  f(\phi_t(y)) dy ds - \int_0^t \int_\T  q_s(\phi^\epsilon_s(y))f(\phi_t(y))dyds \right|} \nonumber
\\
&\qquad\leq
\int_0^t \|q^\epsilon_s-q_s \|_{L^\infty(\T)} ds \| f\|_{L^1(\T)} ;
\end{align}
and
\begin{align} \label{eq:bound4}
&\hE{\left| \int_0^t \int_\T  q_s(\phi^\epsilon_s(y))  f(\phi_t(y)) dy ds - \int_0^t \int_\T  q_s(\phi_s(y))f(\phi_t(y))dyds \right|} \nonumber
\\
&\qquad\leq
\hE{\int_0^t \int_\T  |q_s(\phi^\epsilon_s(y))  -  q_s(\phi_s(y)) | |f(\phi_t(y))|dyds } \nonumber
\\ 
&\qquad\eqqcolon
\hE{\int_0^t \int_\T  |q_s(\phi^\epsilon_s(y))  -  q_s(\phi_s(y)) | d\mu(y)ds},
\end{align}
where $d\mu(y)\coloneqq |f(\phi_t(y))|dy$ is a random Radon measure on $\T$.

By assumptions (A6) and (A7), the terms \eqref{eq:bound1}, \eqref{eq:bound2} and \eqref{eq:bound3} go to zero as $\epsilon \to 0$, using the same reasoning of \cite[Theorem 5.1]{FlPa21}.
Therefore, here we restrict ourselves to only consider the remaining term \eqref{eq:bound4}.

Let us argue \emph{per absurdum}. Suppose by contradiction that there exists a subsequence $\epsilon_k \to 0$ such that 
\begin{align} \label{eq:absurd}
\hE{\int_0^t \int_\T  |q_s(\phi^{\epsilon_k}_s(y))  -  q_s(\phi_s(y)) | d\mu(y)ds}
\geq 
C
\end{align}
for some $C>0$ and for every $\epsilon_k$.

Let $\mathcal{N}$ and $\mathcal{N}$ be negligible sets such that $\phi_t$ is measure preserving for every $\omega \in \mathcal{N}^c$ and $\tilde{\omega} \in \tilde{\mathcal{N}}^c$.

Take $\delta>0$. By Lusin Theorem \cite[Theorem 2.23]{Ru70} there exists a measurable set $C_\delta \subset [0,t] \times \T$ with $\mathscr{L}_{[0,t]} \otimes \mathscr{L}_\T ([0,t] \times \T \setminus C_\delta)<\delta$ and a continuous function $Q_\delta \in C([0,t] \times \T)$ that coincides with $q$ on $C_\delta$. Therefore
\begin{align*}
\int_0^t \int_\T  |q_s(\phi^{\epsilon_k}_s(y))  -  q_s(\phi_s(y)) | &d\mu(y)ds
=
\int_{C_\delta}  |q_s(\phi^{\epsilon_k}_s(y))  -  q_s(\phi_s(y)) | d\mu(y)ds
\\
&\quad+
\int_{[0,t] \times \T \setminus C_\delta}  |q_s(\phi^{\epsilon_k}_s(y))  -  q_s(\phi_s(y)) | d\mu(y)ds
\\
&\leq
\int_{[0,t] \times \T}  |Q_\delta(s,\phi^{\epsilon_k}_s(y))  -  Q_\delta(s,\phi_s(y)) | d\mu(y)ds
\\
&\quad+ 
2\int_{[0,t] \times \T \setminus C_\delta}  \|q_s\|_{L^\infty(\T)}  d\mu(y)ds.
\end{align*}

Let us consider the second term first. 
Recalling $d\mu(y)=|f(\phi_t(y))|dy$, we have
\begin{align*}
\int_{[0,t] \times \T \setminus C_\delta}  \|q_s\|_{L^\infty(\T)}  d\mu(y)ds
&=
\int_{[0,t] \times \T \setminus C_\delta}  \|q_s\|_{L^\infty(\T)}  |f(\phi_t(y))|dyds
\\
&=
\int_{\phi_t^{-1}(C_\delta^c)}  \|q_s\|_{L^\infty(\T)}  |f(y)|dyds,
\end{align*}
with $\phi_t^{-1}(C_\delta^c) \coloneqq \{ (s,y) : (s,\phi_t(y))\in C_\delta^c \}$. Since $\phi_t$ is measure preserving for every $\omega \in \mathcal{N}^c$ and $\tilde{\omega} \in \tilde{\mathcal{N}}^c$, it is easy to check
\begin{align*}
\mathscr{L}_{[0,t]} \otimes \mathscr{L}_\T (\phi_t^{-1}(C_\delta^c)) 
= 
\mathscr{L}_{[0,t]} \otimes \mathscr{L}_\T (C_\delta^c) 
< \delta
\end{align*}
$\hat{\PP}$-almost surely, and since $\|q\|_{L^\infty(\T)}|f| \in L^1([0,t] \times \T)$, absolute continuity of Lebesgue integral gives the existence of  $\delta>0$ such that for every $\omega \in \mathcal{N}^c$ and $\tilde{\omega} \in \tilde{\mathcal{N}}^c$
\begin{align*}
\int_{[0,t] \times \T \setminus C_\delta}  \|q_s\|_{L^\infty(\T)}  d\mu(y)ds < C/3.
\end{align*}

We fix such a $\delta$ hereafter. For the first term we argue as follows: since we have proved
\begin{align*}
 \sup_{t \in [0,T]} \hE{ \| \phi_t^{\epsilon_k}-\phi_t \|_{L^1(\T,\T)}}
\to 0
\end{align*}
as ${\epsilon_k} \to 0$, then for every fixed $s \in [0,T]$ there exists a subsequence (that we still denote $\epsilon_k$) such that the maps 
\begin{align*}
\Phi^{\epsilon_k}_s:\hat{\Omega} \times \T &\to [0,T] \times \T,
\\
\Phi^{\epsilon_k}_s(\hat{\omega},y)&=(s,\phi^\epsilon(\hat{\omega},s,y))
\end{align*}
converge $\hat{\PP} \otimes \mathscr{L}_{\T}$-almost everywhere to $\Phi_s$ given by $\Phi_s(\hat{\omega},y)=(s,\phi(\hat{\omega},s,y))$.
By almost sure continuity in time of $\Phi^{\epsilon_k}_s$ and $\Phi_s$, it is possible to extract a common subsequence $\epsilon_k \to 0$ such that $\Phi^{\epsilon_k}_s$ converges $\hat{\PP} \otimes \mathscr{L}_{\T}$-almost everywhere to $\Phi_s$ simultaneously for all $s \in [0,T]$.   

Therefore, since $Q_\delta$ is continuous on $[0,t] \times \T$, also $Q_\delta(\Phi^{\epsilon_k})$ converges $\hat{\PP} \otimes \mathscr{L}_{[0,t]} \otimes \mathscr{L}_{\T}$-almost everywhere to $Q_\delta(\Phi)$, and since $\mu$ is absolutely continuous with respect to $\mathscr{L}_\T$ for almost every $\hat{\omega} \in \hat{\Omega}$, the convergence is actually $\hat{\PP} \otimes \mathscr{L}_{[0,t]} \otimes \mu_{\hat{\omega}}$-almost everywhere; moreover, $Q_\delta(\Phi^{\epsilon_k})$ is dominated by the constant $\sup_{s \in [0,t], y \in \T} |Q_\delta(s,y)|$, and Lebesgue dominated convergence yields convergence in $L^1(\hat{\Omega} \times [0,T] \times \T,\hat{\PP} \otimes \mathscr{L}_{[0,t]} \otimes \mu_{\hat{\omega}})$, that is
\begin{align*}
\hE{\int_{[0,t] \times \T}  |Q_\delta(s,\phi^{\epsilon_k}_s(y))  -  Q_\delta(s,\phi_s(y)) | d\mu(y)ds} \to 0,
\end{align*}
as ${\epsilon_k} \to 0$. This contradicts \eqref{eq:absurd}, and therefore we have proved: for every $f \in L^1(\T)$
\begin{align*}
\E{\left|\int_\T \Xi^\epsilon_t(x)  f(x)dx
-\int_\T \Xi_t(x)  f(x)dx \right|} \to 0
\quad \mbox{ as } \epsilon \to 0,
\end{align*}
for every fixed $t\in [0,T]$.
Since $\| \Xi^\epsilon_t \|_{L^\infty(\T)}$ is bounded uniformly in $\epsilon>0$ and $t \in [0,T]$, pointwise converges implies convergence in $L^p([0,T])$ for every finite $p$ by Lebesgue dominated convergence Theorem.

Finally, if $q \in L^1([0,T],Lip(\T))$ and $f \in Lip(\T)$ with $[f]_{Lip(\T)} \leq 1$, we have
\begin{align*}
\hE{ \int_\T \left| f(\phi^\epsilon_t(y)) - f(\phi_t(y)) \right| dy} 
&\leq 
\hE{ \int_\T \left| \phi^\epsilon_t(y) - \phi_t(y) \right| dy}
\\
&\leq 
\sup_{t \in [0,T]}  \hE{ \| \phi_t^\epsilon-\phi_t \|_{L^1(\T,\T)}},
\end{align*}
controlling \eqref{eq:bound1} and \eqref{eq:bound2} uniformly in $f$; also, since $\|f\|_{L^\infty(\T)} \leq 1$ it holds
\begin{align*}
&\hE{\int_0^t \int_\T  |q_s(\phi^\epsilon_s(y))  -  q_s(\phi_s(y)) | |f(\phi_t(y))|dyds }
\\
&\leq
\hE{\int_0^t \int_\T  \|q_s\|_{Lip(\T)} 
|\phi^\epsilon_s(y)-\phi_s(y)| dyds }
\\
&\leq
\int_0^t  \|q_s\|_{Lip(\T)} ds 
\sup_{s \in [0,T]}
\hE{ \|\phi^\epsilon_s-\phi_s\|_{L^1(\T,\T)}},
\end{align*}
allowing to bound \eqref{eq:bound4} in a simpler way. Putting all together, we have proved the desired convergence uniformly in $t \in [0,T]$ and $f \in Lip(\T)$ with $[f]_{Lip(\T)} \leq 1$, $\|f\|_{L^\infty(\T)} \leq 1$. The proof is complete.
\end{proof}

\section{Examples} \label{sec:ex}
In this final section, we discuss how assumptions (A1)-(A7) are fullfilled by our main motivational examples, namely advection-diffusion or Navier-Stokes equations at large scales coupled with stochastic Euler equations at small scales - cfr. \autoref{ssec:ex} for details.

First of all, notice that in the case of passive scalars, like in the advection-diffusion equations, there is nothing to actually prove since all the subjects of assumptions (A1)-(A7) are given \emph{a priori}. 
On the other hand, in the Navier-Stokes system the fields $v^\epsilon$, $v$ are given by $v^\epsilon=K \ast \Xi^\epsilon$, $v=K \ast \Xi$, and therefore (A1), (A2) and (A4) need to be checked. 
The verification of (A4) needs an additional requirement on the external source $q$: assume
\begin{itemize}
\item[(\textbf{A8})]
there exists a constant $C$ such that for almost every $t \in [0,T]$ and almost every $x,y \in \T$
\begin{align*}
|q(t,x)-q(t,y)| \leq C \gamma (|x-y|).
\end{align*}
\end{itemize}

\begin{prop} \label{prop:NS}
Let $\nu \geq 0$, $\Xi_0 \in L^\infty(\T)$ with zero spatial average and consider the Navier-Stokes ($\nu > 0$) or Euler ($\nu=0$) system
\begin{align*}
\begin{cases}
d\Xi^\epsilon_t + (v^\epsilon_t + u^\epsilon_t) \cdot \nabla \Xi^\epsilon_t dt
=
\nu \Delta \Xi^\epsilon_t dt + q^\epsilon_t dt,
\\
d \xi^\epsilon_t + (v^\epsilon_t+u^\epsilon_t) \cdot \nabla \xi^\epsilon_t dt
=
- \epsilon^{-1} \xi^\epsilon_t dt + \epsilon^{-1} \sum_{k \in \N} \varsigma_k dW^k_t,
\\
v^\epsilon_t = -\nabla^\perp(-\Delta)^{-1}\Xi^\epsilon_t,
\\
u^\epsilon_t = -\nabla^\perp(-\Delta)^{-1}\xi^\epsilon_t,
\end{cases}
\end{align*}
and the limiting large-scale dynamics
\begin{align*}
\begin{cases}
d \Xi_t + v_t \cdot \nabla \Xi_t dt
+ \sum_{k \in \N} \sigma_k \cdot \nabla \Xi_t \circ dW^k_t
=
\nu \Delta \Xi_t dt + q_t dt,
\\
v_t=-\nabla^\perp(-\Delta)^{-1}\Xi_t .
\end{cases}
\end{align*}

Assume (A3), (A5)-(A8) and take $q^\epsilon_t$, $q_t$ with zero spatial average for almost every $t \in [0,T]$.
Then the velocity fields $v^\epsilon$, $v$ satisfy (A1), (A2) and (A4).
\end{prop}

\begin{proof}
Concerning (A1), measurability can be deduced by $v^\epsilon=K \ast \Xi^\epsilon$, $v=K \ast \Xi$, representation formulas \eqref{eq:repr_large_eps} and \eqref{eq:repr_large}, and the fact that $\phi^\epsilon$, $\phi$ are stochastic flows of measure-preserving homeomorphisms. 
Assumption (A2) is given by $v^\epsilon=K \ast \Xi^\epsilon$, $v=K \ast \Xi$, \eqref{eq:K} and \autoref{lem:log_lip}. 

Finally, let us then verify (A4). Recall
\begin{align*}
v^\epsilon_t(x)
&=
\int_\T K(x-y) \Xi^\epsilon_t(y) dy
\\
&=
\int_\T K(x-y)
\tE{ \Xi_0((\phi^\epsilon_t)^{-1}(y)) + \int_0^t q^\epsilon_s(\phi^\epsilon_s((\phi^\epsilon_t)^{-1}(y)))ds} dy
\\
&=
\tE{ \int_\T K(x-\phi^\epsilon_t(y))
\Xi_0(y) dy} 
+ 
\tE{ \int_\T K(x-\phi^\epsilon_t(y))
\int_0^t q^\epsilon_s(\phi^\epsilon_s(y))ds dy} ,
\end{align*}
and
\begin{align*}
v_t(x)
&=
\int_\T K(x-y) \Xi_t(y) dy
\\
&=
\int_\T K(x-y)
\tE{ \Xi_0((\phi_t)^{-1}(y)) + \int_0^t q_s(\phi_s((\phi_t)^{-1}(y)))ds} dy
\\
&=
\tE{ \int_\T K(x-\phi_t(y))
\Xi_0(y) dy} 
+ 
\tE{ \int_\T K(x-\phi_t(y))
\int_0^t q_s(\phi_s(y))ds dy} .
\end{align*}

We have
{\small
\begin{align*}
\int_\T |v^\epsilon_t(x)-v_t(x)| dx
&\leq
\tE{
\int_\T 
\int_\T 
\left| K(x-\phi^\epsilon_t(y)) -  K(x-\phi_t(y)) \right|
|\Xi_0(y)| dy dx
}
\\
&\hspace{-3cm}\quad+
\tE{
\int_\T \left|
\int_\T K(x-\phi^\epsilon_t(y))
\int_0^t q^\epsilon_s(\phi^\epsilon_s(y))ds dy
-
\int_\T K(x-\phi_t(y))
\int_0^t q_s(\phi_s(y))ds dy
\right| dx }
\\
&\leq
\tE{
\int_\T 
\int_\T 
\left| K(x-\phi^\epsilon_t(y)) -  K(x-\phi_t(y)) \right|
|\Xi_0(y)| dy dx
}
\\
&\quad+
\tE{
\int_\T 
\int_\T \left|
K(x-\phi^\epsilon_t(y))-K(x-\phi_t(y)) \right|
\left| \int_0^t q^\epsilon_s(\phi^\epsilon_s(y))ds \right| dy dx
}
\\
&\quad+
\tE{
\int_\T 
\int_\T \left|
K(x-\phi_t(y)) \right|
 \int_0^t \left|q^\epsilon_s(\phi^\epsilon_s(y))-
q_s(\phi^\epsilon_s(y))ds \right| dy dx
}
\\
&\quad+
\tE{
\int_\T 
\int_\T \left|
K(x-\phi_t(y)) \right|
 \int_0^t \left|q_s(\phi^\epsilon_s(y))-
q_s(\phi_s(y))ds \right| dy dx
}
\\
&\lesssim
\gamma \left( \tE{\|\phi^\epsilon_t-\phi_t \|_{L^1(\T,\T)}}\right)
+
\int_0^t \|q^\epsilon_s - q_s \|_{L^\infty(\T)} ds
\\
&\quad+
\int_0^t \gamma \left( \tE{\|\phi^\epsilon_s-\phi_s \|_{L^1(\T,\T)}}\right) ds,
\end{align*}}
that is the desired estimate, since $\int_0^t \|q^\epsilon_s - q_s \|_{L^\infty(\T)} ds \to 0$ as $\epsilon \to 0$ by assumption (A7).
\end{proof}

\bibliographystyle{plain}

\begin{thebibliography}{1}

%
%
%

\bibitem{BoEc12}
Guido Boffetta and Robert~E. Ecke.
\newblock Two-dimensional turbulence.
\newblock {\em Annual Review of Fluid Mechanics}, 44(1):427--451 (2012).

\bibitem{BrCaFl90} Z. Brze\'{z}niak, M. Capinski, and F. Flandoli, Approximation
for diffusion in random fields. {\em Stoch. Anal. Appl.} 8 (1990), 293--313.

\bibitem{BCF 91 mult noise} Z. Brze\'{z}niak, M. Capinski, and F. Flandoli.
Stochastic partial differential equations and turbulence. {\em Math. Models Methods Appl. Sci.} 1 (1991), no. 1, 41--59.

\bibitem{BCF 92 mult noise} Z. Brze\'{z}niak, M. Capinski, and F. Flandoli.
Stochastic {N}avier-{S}tokes equations with multiplicative noise. {\em Stochastic
Anal. Appl.} 10 (1992), no. 5, 523--532.

\bibitem{BrFl95} Z. Brze\'{z}niak and F. Flandoli. 
Almost  sure  approximation  of Wong-Zakai  type 
for  stochastic  partial  differential  equations. {\em Stochastic Process. Appl.} 55 (1995), no. 2, 329--358.

\bibitem{BrFlMa16}
Z.~Brzeźniak, F.~Flandoli, and M.~Maurelli.
\newblock Existence and uniqueness for stochastic 2{D} {E}uler flows with
  bounded vorticity.
\newblock {\em Arch. Rational Mech. Anal.}, 221:107--142 (2016).

\bibitem{BrSl20+} Z. Brze\'{z}niak and J. Slavik. Well-posedness of the 3{D}
stochastic primitive equations with transport noise, {arXiv:2008.00274}.

\bibitem{CoIy08} P. Constantin and G. Iyer.
A stochastic Lagrangian representation of the three-dimensional incompressible Navier-Stokes equations. {\em Comm. Pure Appl. Math.} 61:3 (2008), 330--345.

\bibitem{CrFlHo19} D. Crisan, F. Flandoli and D. D. Holm. Solution properties
of a 3{D} stochastic {E}uler fluid equation. {\em J. Nonlinear Sci.} 29 (2019),
813--870.

\bibitem{Cruzeiro Torr} A. B. Cruzeiro and I. Torrecilla. On a 2{D} stochastic {E}uler equation of transport type: existence and geometric formulation.
{\em Stoch. Dyn.} 15 (2015), no. 1, 1450012, 19 pp.

\bibitem{Delarue} F. Delarue. Restoring uniqueness to mean-field games by
randomizing the equilibria. {\em Stoch. Partial Differ. Equ. Anal. Comput.} 7
(2019), no. 4, 598--678.

\bibitem{Dolgo} D. Dolgopyat, V. Kaloshin, and L. Koralov. Sample path
properties of the stochastic flows. {\em Ann. Probab.} 32 (2004), no. 1A, 1--27.

\bibitem{DrivasHolm} T. D. Drivas and D. D. Holm. Circulation and energy
theorem preserving stochastic fluids. {\em Proc. Roy. Soc. Edinburgh Sect. A} 150
(2020), no. 6, 2776--2814.

\bibitem{DrivasHolmLehaly} T. D. Drivas, D. D. Holm and J.-M. Leahy. {L}agrangian
averaged stochastic advection by {L}ie transport for fluids. {\em J. Stat. Phys.}
179 (2020), no. 5-6, 1304--1342.

\bibitem{Fla libro} F. Flandoli. Stochastic Partial Differential Equations
in Fluid Mechanics, Springer, to appear.

\bibitem{FlaGaleLuoJEE} F. Flandoli, L. Galeati, and D. Luo. Scaling limit of
stochastic 2{D} {E}uler equations with transport noises to the deterministic
{N}avier-{S}tokes equations. {\em J. Evol. Equ.} 21 (2021), no. 1, 567-600.

\bibitem{FlaGaleLuoCPDE} F. Flandoli, L. Galeati, and D. Luo. Delayed blow-up by
transport noise. {\em Comm. Partial Differential Equations} (2021),
https://doi.org/10.1080/03605302.2021.1893748.

\bibitem{FlaGaleLuoPTRSA} F. Flandoli, L. Galeati, and D. Luo. Eddy heat
exchange at the boundary under white noise turbulence, to appear on
{\em Philosoph. Trans. A. Royal Soc.}

\bibitem{FlaGalLuorate} F. Flandoli, L. Galeati, and D. Luo. Mixing, dissipation
enhancement and convergence rates for scaling limit of {S}{P}{D}{E}s with transport
noise, arXiv:2104.01740.

\bibitem{FlGuPr10} F. Flandoli, M. Gubinelli, and E. Priola.
Well-posedness of the transport equation
by stochastic perturbation. {\em Invent. math.} 180 (2010), 1–53.

\bibitem{FlaMaurNek} F. Flandoli, M. Maurelli, and M. Neklyudov. Noise prevents
infinite stretching of the passive field in a stochastic vector advection
equation. {\em J. Math. Fluid Mech.} 16 (2014), no. 4, 805-822.

\bibitem{FlaOliv advection} F. Flandoli and C. Olivera. Well-posedness of the
vector advection equations by stochastic perturbation. {\em J. Evol. Equ.} 18
(2018), no. 2, 277-301.

\bibitem{FlaPappWater} F. Flandoli and U. Pappalettera. Stochastic modeling of
small scale perturbation. {\em Water} 12 (2020), 10, 2950.

\bibitem{FlPa21} F. Flandoli and U. Pappalettera. \newblock2{D} {E}%
uler equations with {S}tratonovich transport noise as a large-scale
stochastic model reduction. \newblock {\em J. Nonlinear Sci.}, 31:24 (2021).

\bibitem{Funaki} T. Funaki. A. Inoue, On a new derivation of the
Navier-Stokes equation. {\em Comm. Math. Phys.} 65 (1979), 1, 83-90.

\bibitem{Galeati} L. Galeati. On the convergence of stochastic transport
equations to a deterministic parabolic one. {\em Stoch. Partial Differ. Equ.
Anal. Comput.} 8 (2020), no. 4, 833-868.

\bibitem{Gess} B. Gess and I. Yaroslavtsev. Stabilization by transport noise
and enhanced dissipation in the Kraichnan model, arXiv:2104.03949.

\bibitem{Gy88} I. Gyongy. On the approximation of stochastic partial
differential equations i. {\em Stochastics} 25 (1988), 59-85.

\bibitem{Gy89} I. Gyongy. On the approximation of stochastic partial
differential equations ii. {\em Stochastics} 26 (1989), 129-164.

\bibitem{JiZh20} C. Jia and G. Zhao. \newblock Moderate maximal
inequalities for the {O}rnstein-{U}hlenbeck process. \newblock {\em Proc.
Amer. Math. Soc.}, 148:3607--3615 (2020).

\bibitem{HoLeNi19} M. Hofmanova, J. Leahy, and T. Nilssen. On the {N}avier-{S}%
tokes equations perturbed by rough transport noise. {\em J. Evol. Eq. 19} (2019),
203--247.

\bibitem{HoLeNi19+} M. Hofmanova, J. Leahy, and T. Nilssen. On a rough perturbation of the {N}avier-{S}tokes system and its vorticity formulation, {arXiv:1902.09348}.

\bibitem{Holm} D. D. Holm. Variational principles for stochastic fluid
dynamics. {\em Proc. R. Soc. A.} 471 (2015), 20140963.

\bibitem{Krause Radler} F. Krause and K.-H. R\"{a}dler. Mean Field Magnetohydrodynamics and Dynamo Theory. Pergamon Press, Oxford 1980.

\bibitem{LeJan Ray} Y. Le Jan and O. Raimond. Integration of Brownian vector
fields. {\em Ann. Probab.} 30 (2002), no. 2, 826--873.

\bibitem{MajdaK} A.J. Majda and P.R. Kramer. Simplified models for turbulent
diffusion: Theory, numerical modelling, and physical phenomena. {\em Physics
Reports} 314 (1999).

\bibitem{MaTiVE01} A.J. Majda, I. Timofeyev, and E. Vanden~Eijnden. A
mathematical framework for stochastic climate models. {\em Comm. Pure Appl. Math.}
54 (2001), 891--974.

\bibitem{MaPu94}
C. Marchioro and M. Pulvirenti.
\newblock {\em Mathematical theory of incompressible nonviscous fluids},
  volume~96 of {\em Applied Mathematical Sciences}.
\newblock Springer-Verlag, New York, 1994.

\bibitem{MiRo04} R. Mikulevicius and B.L. Rozovskii. Stochastic {N}avier-{S}%
tokes equations for turbulent flows. {\em SIAM J. Math. Anal.} 35 (2004),
1250-1310.

\bibitem{MiRo05} R. Mikulevicius and B.L. Rozovskii. Global $L^2$-solutions of
stochastic {N}avier-{S}tokes equations. {\em Ann. Probab.} 33 (2005), 137--176.

\bibitem{Pa21+} U. Pappalettera. \newblock Quantitative mixing and
dissipation enhancement property of {O}rnstein-{U}hlenbeck flow, arXiv:2104.03732.

\bibitem{Ru70}
W. Rudin. 
\newblock {\em Real and complex analysis}.
\newblock MGH, 1970.

\bibitem{Sreen} K. R. Sreenivasan. Turbulent mixing: A perspective. {\em PNAS} 116
(2019), no. 37.

\bibitem{TeZa06} G. Tessitore and J. Zabczyk. Wong-{Z}akai approximations of
stochastic evolution equations. {\em J. Evol. Eq.} 6 (2006), 4, 621-655.

\bibitem{Tw93} K. Twardowska. Approximation theorems of {W}ong-{Z}akai type
for stochastic differential equations in infinite dimensions. {\em Diss. Math.
(Rozprawy Mat.)} 325 (1993).

\bibitem{Yokohama} S. Yokoyama. Construction of weak solutions of a certain
stochastic Navier-Stokes equation. {\em Stochastics} 86 (2014), no. 4, 573-593.

\bibitem{Zeldovich} Ya. B. Zeldovich, S. A. Molchanov, A. A. Ruzmaikin, and
D. D. Sokolov. Intermittency in random media. {\em Sov. Phys. Usp.} 30 (1987), 5,
353-369.

\end{thebibliography}

\end{document}